\documentclass[11pt,a4paper]{amsart}
\usepackage{amsmath, amssymb}
\usepackage{amsfonts}
\usepackage[arrow,matrix,curve,cmtip,ps]{xy}
\usepackage{amsthm}
\usepackage{amscd}
\usepackage{hyperref} 
\usepackage[dvipsnames]{xcolor}
\usepackage{pst-node}
\usepackage{tikz-cd} 

\allowdisplaybreaks

\newtheorem{theorem}{Theorem}[section]
\newtheorem{lemma}{Lemma}[section]
\newtheorem{proposition}{Proposition}[section]
\newtheorem{corollary}{Corollary}[section]

\theoremstyle{definition}
\newtheorem{definition}{Defintion}[section]
\newtheorem{notation}[definition]{Notation}
\newtheorem{convention}[definition]{Convention}
\theoremstyle{remark}
\newtheorem{remark}{Remark}[section]

\numberwithin{equation}{subsection}
\numberwithin{theorem}{section}
\newcount\colveccount
\newcommand*\colvec[1]{
	\global\colveccount#1
	\begin{pmatrix}
		\colvecnext
	}
	\def\colvecnext#1{
		#1
		\global\advance\colveccount-1
		\ifnum\colveccount>0
		\\
		\expandafter\colvecnext
		\else
	\end{pmatrix}
	\fi
}

\newcommand{\Go}{\mathsf{SO}_0(n+1,n)}
\newcommand{\Ga}{\mathsf{SO}_0(n+1,n)\ltimes\R^{2n+1}}

\newcommand{\sv}{\mathsf{v}}
\newcommand{\sw}{\mathsf{w}}

\newcommand{\sad}{\mathsf{Ad}}
\newcommand{\adj}{\mathsf{adj}}
\newcommand{\tr}{\mathsf{Tr}}
\newcommand{\sA}{\mathsf{A}}
\newcommand{\sH}{\mathsf{H}}
\newcommand{\sP}{\mathsf{P}}
\newcommand{\spr}{\mathsf{pr}}
\newcommand{\sQ}{\mathsf{Q}}
\newcommand{\sT}{\mathsf{T}}
\newcommand{\sth}{\mathsf{t}}
\newcommand{\sG}{\mathsf{G}}

\newcommand{\sM}{\mathsf{M}}

\newcommand{\cB}{\mathcal{B}}
\newcommand{\cD}{\mathcal{D}}

\newcommand{\cP}{\mathcal{P}}
\newcommand{\cH}{\mathcal{H}}
\newcommand{\cQ}{\mathcal{Q}}

\newcommand{\cG}{\mathcal{G}}
\newcommand{\cR}{\mathcal{R}}
\newcommand{\cO}{\mathcal{O}}
\newcommand{\cM}{\mathcal{M}}
\newcommand{\cL}{\mathcal{L}}

\newcommand{\R}{\mathbb{R}}

\newcommand{\st}{\mathsf{top}}
\newcommand{\sHom}{\mathsf{Hom}}

\newcommand{\defeq}{\mathrel{\mathop:}=}

\newcommand{\flow}{\mathsf{U}\Gamma}
\newcommand{\cflow}{\widetilde{\mathsf{U}_0\Gamma}}
\newcommand{\bdry}{\partial_\infty\Gamma}
\newcommand{\fg}{\mathfrak{g}}
\newcommand{\fh}{\mathfrak{h}}
\newcommand{\ft}{\mathfrak{t}}
\newcommand{\fk}{\mathfrak{k}}
\newcommand{\fu}{\mathfrak{u}}
\newcommand{\fF}{\mathfrak{F}}

\title[Margulis multiverse]{Margulis multiverse: infinitesimal rigidity,\\ pressure form and convexity}
\author{Sourav Ghosh}
\address{Department of Mathematics, Luxembourg and Ashoka University}
\email{sourav.ghosh@ashoka.edu.in, sourav.ghosh.bagui@gmail.com}
\date{\today}

\thanks{The author acknowledges support from the AGoLoM
	FNR grant OPEN/16/11405402 at University of Luxembourg and Ashoka University annual research grant.}
\begin{document}
	
	\begin{abstract}
		In this article we construct the pressure forms on the moduli spaces of higher dimensional Margulis spacetimes without cusps and study their properties. We show that the Margulis spacetimes are infinitesimally determined by their marked Margulis invariant spectra. We use this fact to show that the restrictions of the pressure form give Riemannian metrics on the constant entropy sections of the quotient moduli space. We also show that constant entropy sections of the moduli space with fixed linear parts bound convex domains.
	\end{abstract}
	
	\maketitle
	\tableofcontents
	
	\newpage
	
	\section*{Introduction}

	The goal of this article is to understand the geometry of the moduli space of Margulis spacetimes. Margulis spacetimes are examples of proper affine actions of free groups. The study of proper affine actions stems from the classification results of Bieberbach \cite{B1,B2} for Euclidean crystallographic groups (i.e. subgroups of  $\mathsf{O}(m)\ltimes\R^m$ which act properly discontinuously, freely and cocompactly on $\R^m$) (for more details see \cite{F1,Aus2,Buser}).  Biebarbach's results show that Euclidean crystallographic groups are virtually abelian. Later, in \cite{Aus} Auslander attempted at a proof that affine crystallographic groups (i.e. subgroups of $\mathsf{GL}(m)\ltimes\R^m$ which act properly discontinuously, freely and cocompactly on $\R^m$) are virtually solvable. The proof turned out to be incomplete and the statement was rechristened as the Auslander conjecture. The conjecture is still open but it has been shown to hold true for the three dimensional case by Fried--Goldman \cite{FG} and for dimension less than 7 by Abels--Margulis--Soifer \cite{AMS2}. Moreover, in \cite{Milnor} Milnor wondered if the cocompactness assumption in the Auslander conjecture could be dropped. Answering Milnor's question, Margulis \cite{Margulis1,Margulis2} showed that if the cocompactness assumption is dropped from the Auslander conjecture then the conjecture no longer holds true. In particular, he showed the existence of proper affine actions of non-abelian free groups on $\R^3$. He used the north-south dynamics of Schottky actions to prove the existence of such actions. These examples were a surprise to the community due to the counter intuitiveness of its actions. An interesting property of proper affine actions observed by Fried--Goldman \cite{FG} states that the Zariski closure of the linear part of such an action necessarily lies in some conjugate of $\mathsf{SO}(2,1)$. Later, Drumm \cite{Drumm2} classified the linear parts of such actions on $\R^3$. He showed that any free finitely generated discrete subgroup of $\mathsf{SO}_0(2,1)$ admits affine deformations which act properly on $\R^3$. In this article, we will only be interested in the case where the linear parts of the non-abelian free subgroups of the affine group, which give rise to proper actions, do not contain any parabolic elements.
	
	These proper actions admit very interesting fundamental domains. In \cite{Drumm1} Drumm constructed explicit fundamental domains for a large class of them whose boundary consists of the so called crooked planes. Recently, it was shown by Charrette--Drumm--Goldman \cite{CDG1} for the case of two generator Fuchsian groups and by Danciger--Gu\'eritaud--Kassel \cite{DGK2} for the general case that any such proper affine action admits a fundamental domain bounded by crooked planes. 
	
	The construction of proper affine actions due to Margulis was later generalized by Abels--Margulis--Soifer \cite{AMS}. Let $\cQ$ be a non-degenerate quadratic form on $\sT:=\R^{2n+1}$ of signature $(n+1,n)$ and let $\sH\defeq\Go$ be the connected component containing identity of its isometry group. They showed, using a proximality argument, that for odd $n$ the group $\sG:=\sH\ltimes\sT$ admits non-abelian free subgroups which act properly discontinuously and freely on $\sT$ and for even $n$ the group $\sG$ does not admit any such subgroups. Subsequently but out of completely different motivations, Labourie \cite{Labourie} came up with the notion of an Anosov representation. Furthermore, it was realized by Guichard--Wienhard in \cite{GW2}, that the notion of an Anosov representation is closely related with the notion of proximality used in \cite{AMSproximal,AMS}. It follows that the examples constructed by Abels--Margulis--Soifer have Anosov linear parts. In this article we generalize the results of \cite{Ghosh1} and construct the pressure form on the moduli space of such proper actions and study its properties. 
	
	Let $\Gamma$ be a word hyperbolic group and let $\sP$ be the stabilizer of a maximal isotropic plane under the linear action of $\sH$ on $\sT$. Let $\rho:\Gamma\to\sG$ be an injective homomorphism. We denote its \emph{linear part} by $L_\rho$ and its \emph{translational part} by $u_\rho$ i.e. $L_\rho:\Gamma\to\sH$, $u_\rho:\Gamma\to\sT$ and $\rho=(L_\rho,u_\rho)\in\sH\ltimes\sT=\sG$. We say $\rho$ is a Margulis spacetime if and only if its linear part $L_\rho$ is $\sP$-Anosov (see subsection \ref{ss:ano} for the definition) in $\sH$ and $\rho(\Gamma)$ acts properly on $\sT$. We denote the space of $\sP$-Anosov, Zariski dense representations in $\sH$ whose conjugacy classes are also smooth points of the representation variety $\sHom(\Gamma,\sH)/\sH$ by $\sHom_{\sA}(\Gamma,\sH)$. We denote the space of representations $\rho:\Gamma\rightarrow\sG$ whose linear part is $\sHom_{\sA}(\Gamma,\sH)$ by $\sHom_{\sA}(\Gamma,\sG)$.	
	In this article, we only consider Margulis spacetimes $\rho\in\sHom_{\sA}(\Gamma,\sG)$ and denote the space of such Margulis spacetimes by $\sHom_{\sM}(\Gamma,\sG)$. The natural projection map  $L:\sHom_{\sA}(\Gamma,\sG)\rightarrow\sHom_{\sA}(\Gamma,\sH)$ which takes a representation to its linear part endows $(\sHom_{\sA}(\Gamma,\sG),\sHom_{\sA}(\Gamma,\sH),L)$ with a vector bundle structure. We observe that $\sG$ acts on $\sHom_{\sM}(\Gamma,\sG)$ via conjugation and define \[\sM(\Gamma,\sG)\defeq\sHom_{\sM}(\Gamma,\sG)/\sG.\]
	We note that $\sM(\Gamma,\sG)$ can be empty. If $\Gamma$ is free, that happens if and only if $n$ is even. We assume from now on that $\Gamma$ and $\sG$ are such that $\sM(\Gamma,\sG)$ is non-empty.
	
	In \cite{Margulis1,Margulis2} Margulis also introduced the notion of a Margulis invariant to detect proper affine actions. Suppose $g=(h,u)\in\sG=\sH\ltimes\sT$ is such that $h$ is pseudo-hyperbolic. Then $h$ has a one dimensional eigenspace of unit eigenvalue. Hence there are exactly two eigenvectors of unit eigenvalue and unit norm. One can consistently choose one of these two eigenvectors for all pseudo-hyperbolic elements such that certain orientation convention is preserved and define the Margulis invariant $\alpha(g)$ as the length of the projection of $u$ in the direction of this eigenvector (for more details please see Definition \ref{def:marginv}). Margulis showed that the Margulis invariant spectrum of a Margulis spacetime cannot change sign and conjectured that the converse is also true. A slightly modified version of the converse was proved in the Fuchsian case by Goldman--Labourie--Margulis \cite{GLM} and in the general case by Ghosh--Treib \cite{GT}. They used a diffused version of the Margulis invariant, called the Labourie--Margulis invariant, first introduced by Labourie in \cite{Labourie2} and showed that a representation in $\sHom_{\sA}(\Gamma,\sG)$ is a Margulis spacetime if and only if its Labourie--Margulis invariant is non-vanishing. 
	
	In this article, we use this result and the stability of Anosov representations from \cite{Labourie,GW2} to prove the following:
	\begin{proposition}
		Suppose $\sM(\Gamma,\sG)$ is the space of all Margulis space times as defined above. Then $\sM(\Gamma,\sG)$ is an open subset of the character variety and it is a fibered space. Moreover, each fiber is a disjoint union of two open convex cones in some vector space which differ by a sign.
	\end{proposition}
	
	Furthermore, we use the non-vanishing Labourie-Margulis invariant of a Margulis spacetime, the Anosov structure on its linear part and the theory of thermodynamic formalism developed by Bowen, Ruelle, Parry and Pollicott to show that the growth rate of the number of closed orbits whose Margulis invariants are less than some variable is exponential in that variable. In fact, we also show that the Labourie--Margulis invariant varies analytically over the moduli space of Margulis spacetimes. Suppose $\rho\in\sHom_{\sM}(\Gamma,\sG)$ and let $\cO$ denote the set of periodic orbits of the Gromov flow space (please see Remark \ref{rem:gromovflow}). We denote, $R_T(\rho)\defeq\{\gamma\in\cO\mid\alpha(\rho(\gamma))\leqslant T\}$ and consider the following limit,
	\[h_\rho=\lim_{T\to\infty}\frac{1}{T}\log|R_T(\rho)|.\]
	Also, for $\rho_1\in\sHom_{\sA}(\Gamma,\sG)$ we consider
	\[I(\rho,\rho_1)=\lim_{T\to\infty}\frac{1}{|R_T(\rho)|}\sum_{\gamma\in R_T(\rho)}\frac{\alpha_{\rho_1}(\gamma)}{\alpha_\rho(\gamma)}.\]
	In fact, we prove the following:
	\begin{lemma}
		 Suppose $\rho\in\sHom_{\sM}(\Gamma,\sG)$ and $\rho_1\in\sHom_{\sA}(\Gamma,\sG)$. Then $h_\rho$ and $I(\rho,\rho_1)$ exist and are finite. Moreover, $h_\rho$ is positive and the maps $h,I$ are analytic.
	\end{lemma}
		Now applying the theory of thermodynamic formalism for the analytic maps $h$ and $I$, we define a pressure form on the quotient moduli space. Suppose $\{\rho_t\}_{t\in\R}\subset\sHom_{\sM}(\Gamma,\sG)$ is an analytic one parameter family such that $\rho_0=\rho$ and ${U}(\gamma)=\left.\frac{d}{dt}\right|_{t=0}\rho_t(\gamma)\rho(\gamma)^{-1}$ for all $\gamma\in\Gamma$. Let $\fg$ denote the Lie algebra of $\sG$. We observe that $[{U}]\in\mathsf{H}^1_{\sad\circ\rho}(\Gamma,\fg)\cong \sT_{[\rho]}\sM(\Gamma,\sG)$ (please see Remark \ref{rem:dermarg}) and define
		\[\spr_{[\rho]}([{U}],[{U}])\defeq\left.\frac{d^2}{dt^2}\right|_{t=0}\frac{h_{\rho_t}}{h_\rho}I(\rho,\rho_t)\]	
	Moreover, we show that:
	\begin{proposition}\label{intprop:prm}
		The map $\spr:\sT{\sM}(\Gamma,\sG)\times\sT{\sM}(\Gamma,\sG)\rightarrow\R$ is well defined and it defines a positive semi-definite bilinear form. Moreover, if
		\[\left.\frac{d}{dt}\right|_{t=0}h_{\rho_t}=0,\]
		then $\spr_{[\rho]}([{U}],[{U}])=0$ implies $[{U}]=0\in\mathsf{H}^1_{\sad\circ\rho}(\Gamma,\fg)\cong\sT_{[\rho]}\sM(\Gamma,\sG)$.
		
	\end{proposition}
	In fact, while characterizing the norm zero vectors of the pressure form in Proposition \ref{intprop:prm}, we use an infinitesimal marked Margulis invariant spectrum rigidity result for Margulis spacetimes. Indeed, we introduce a new algebraic expression for the Margulis invariants involving adjugates of the linear parts. Let $\mathfrak{gl}(2n+1,\R)$ be the space of all $(2n+1)\times(2n+1)$ matrices, let $\alpha(g)$ be the Margulis invariant of any $g=(h,u)\in\sG=\sH\ltimes\sT$ whose linear part $h$ is pseudo-hyperbolic and let $\adj(h)$ be the adjugate of $h\in\mathfrak{gl}(2n+1,\R)$ (for more details please see Definition \ref{def:adjugate}; also see Section 0.8.2 of \cite{HJ}).
	\begin{lemma}
		 Suppose $g=(h,u)\in\sH\ltimes\sT=\sG$ with $h$ pseudo-hyperbolic. Then the square of the Margulis invariant can be written as a rational function in $g$ as follows:
		\[\alpha(g)^2=\frac{\langle[\adj(I-h)]u\mid u\rangle}{\tr[\adj(I-h)]}.\]
	\end{lemma}
	Suppose $(g,G)\in\sG\ltimes_\sad\fg$ and  $\{g_t\}_{t\in(-1,1)}\in\sG$ is an analytic one parameter family such that $G=\left.\frac{d}{dt}\right|_{t=0}g_tg^{-1}$. We assume that the respective linear parts of $g_t$ are pseudo-hyperbolic. Then we define
	\[\dot{\alpha}(g,G)\defeq\left.\frac{d}{dt}\right|_{t=0}\alpha(g_t).\]
	We use the algebraic property of the Margulis invariant and a modified version of the Margulis type argument for marked length spectrum rigidity, to prove the following result:
	\begin{theorem}
		Suppose $(\rho,{U})\in\sHom(\Gamma,\sG\ltimes_\sad\fg)$ such that $\rho\in\sHom_{\sM}(\Gamma,\sG)$.
		Then there exists $G\in\fg$ such that ${U}(\gamma)=\sad(\rho(\gamma))(G)-G$ for all $\gamma\in\Gamma$ if and only if for all $\gamma\in\Gamma$ the following holds:
		\[\dot{\alpha}(\rho(\gamma),U(\gamma))=0.\]
	\end{theorem}
	The above Theorem is an infinitesimal version of the marked Margulis invariant spectrum rigidity results obtained by Drumm--Goldman \cite{DG}, Charette--Drumm \cite{CD} and Kim \cite{Kim}. They proved that the marked Margulis invariant spectrum of a Margulis spacetime uniquely determines the Margulis spacetime up to conjugacy.
	
	Finally, we use Proposition \ref{intprop:prm} and observe that the constant entropy sections play a very important role in the study of the geometry of the quotient moduli space. Hence, in the last part of our article we study the constant entropy sections in more detail. We call sections of the quotient moduli space with constant entropy $k$ to be the \textit{Margulis} \textit{multiverse} of entropy $k$ and denote them repectively by $\sM(\Gamma,\sG)_k$. We show that,
	\begin{lemma}\label{intlem:ansub}
		Suppose $\sM(\Gamma,\sG)_k$ are the constant entropy sections of $\sM(\Gamma,\sG)$. Then $\sM(\Gamma,\sG)_k$ is an analytic submanifold of $\sM(\Gamma,\sG)$ of codimension one and
		\[\sM(\Gamma,\sG)_1=\{(L_\rho,h(\rho)u_\rho)\mid\rho\in\sM(\Gamma,\sG)\}.\]
	\end{lemma}
	Moreover, we combine Lemma \ref{intlem:ansub} and Proposition \ref{intprop:prm} to obtain the following result describing the pressure form:
	\begin{theorem}
		Suppose $\sM(\Gamma,\sG)_k$ are the constant entropy sections of $\sM(\Gamma,\sG)$. Then the restriction of the pressure form 
		\[\spr:\sT\sM(\Gamma,\sG)_k\times\sT\sM(\Gamma,\sG)_k\rightarrow\R\] 
		is a Riemannian metric for all $k>0$ and the pressure form on $\sM(\Gamma,\sG)$ is positive semi-definite with rank $\dim(\sM(\Gamma,\sG))-1$. 
	\end{theorem}
	In fact, we give a description of the degenerate direction of the pressure form too. Lastly, we conclude our article by showing a convexity result for the Margulis multiverses. We show that the constant entropy sections with fixed linear part bound a convex domain.
	\begin{theorem}
		Suppose $\varrho\in\sHom_{\sA}(\Gamma,\sH)$ and $u_0,u_1:\Gamma\to\sT$ are such that $(\varrho,u_0),(\varrho,u_1)\in\sHom_{\sM}(\Gamma,\sG)$ with $h(\varrho,u_0)=k=h(\varrho,u_1)$. Then for all $t\in(0,1)$ the representations $(\varrho,(1-t)u_0+tu_1)\in\sHom_{\sM}(\Gamma,\sG)$ and the following holds: 
		\[h(\varrho,(1-t)u_0+tu_1)< k.\]
	\end{theorem}
	Similar convexity results of this nature but in different contexts were obtained by Quint in \cite{Q} and by Sambarino in \cite{Samba1}.
	
	\section*{Acknowledgements}
	
	I would like to thank Prof. Fran\c cois Labourie for explaining to me the rigidity argument due to Prof. Gregory Margulis and the anonymous referee for many helpful suggestions. I would also like to thank the University of Luxembourg for hosting me while doing this research.

	\section{Preliminaries}
	
	In this section we give an overview of the theory of Thermodynamic formalism, Anosov representations and Margulis spacetimes. Our overview will be minimal and will only state results which play an essential role in proving the results of this article.

	\subsection{Thermodynamic formalism}\label{ss:tf}
	
	In this subsection we provide a brief overview of the theory of thermodynamic formalism. Thermodynamic formalism is the study of dynamical systems using intuition coming from Thermodynamics of Physics. It was developed by Bowen, Parry--Pollicott, Pollicott, Ruelle and others. In Section \ref{sec:pfc}, we will use this theory to define and study a quadratic form $\spr$ on $\sM(\Gamma,\sG)$, called the pressure form. 
	
	\begin{definition}
		Suppose $X$ is a metric space along with a flow $\phi$. Then $\phi$ is called topologically transitive if and only if given any two open sets $U,V\subset X$, there exists $t\in\R$ such that $\phi_t(U)\cap V$ is non-empty.
	\end{definition}
	\begin{definition}
		A lamination $\cL$ on $X$ is a partition of $X$ such that there exist two topological spaces $U_1, U_2$ and for all $p\in X$ there exist an open set $O_p$ and a homeomorphism $h_p\defeq(h_p^1,h_p^2):O_p\rightarrow U_1\times U_2$ satisfying the following properties:
		\begin{enumerate}
			\item For all $z,w\in O_p\cap O_q$: $h_p^1(z)=h_p^1(w)$ if and only if $h_q^1(z)=h_q^1(w)$,
			\item $p\cL q$ if and only if there exist $\{p_i\}_{i=0}^n$ with $p_0=p$ and $p_n=q$ such that $p_{i+1}\in O_{p_i}$ and $h_{p_i}^1(p_i)=h_{p_i}^1(p_{i+1})$ for all $i$.
		\end{enumerate}
	\end{definition}
	\begin{definition}
		Suppose $X$ is a compact metric space along with a flow $\phi$. Then $\phi$ is called metric Anosov if and only if there exist two laminations $\cL^\pm$ satisfying the following properties:
		\begin{enumerate}
			\item The laminations $\cL^+$ and $\cL^-$ are transverse to the flow lines,
			\item Both $(\cL^-,\cL^{+,0})$ and $(\cL^+,\cL^{-,0})$ define a local product structure, where for all $x\in X$, 
			\[\cL^{+,0}_x\defeq\cup_{t\in\R}\phi_t\cL^+_x \text{ and } \cL^{-,0}_x\defeq\cup_{t\in\R}\phi_t\cL^-_x,\]
			\item $\cL^+$ (respectively $\cL^-$) is contracted by the flow (respectively by the inverse flow).
		\end{enumerate}
	\end{definition}
	\begin{remark}
		We note that the above property of a flow on a compact metric space was first introduced under the name Smale flow by Pollicott in \cite{Pol}. Later, these flows were adapted in \cite{BCLS} under the name metric Anosov flows, to study projective Anosov representations. In this article, we use the definition given in \cite{BCLS}.
	\end{remark}
	
	\begin{definition}
		Let $X$ be a compact metric space along with a H\"older continuous flow $\phi$ which has no fixed points and let $f,g:X\rightarrow\R$ be two H\"older continuous functions. Then $f$ and $g$ are called Liv\v sic cohomologous if and only if there exists a H\"older continuous function $h:X\rightarrow\R$ which is differentiable along flow lines of $\phi$ and the following holds:
		\[f-g=\left.\frac{\partial}{\partial t}\right|_{t=0}h\circ\phi_t.\]
	\end{definition}
	\begin{definition}
		Suppose $f:X\rightarrow\R$ is a H\"older continuous function then the Liv\v sic cohomology class $[f]$ is called positive (respectively negative) if and only if there exists a H\"older continuous function $g:X\rightarrow\R$ such that $g>0$ (respectively $g<0$) and $[f]=[g]$. 
	\end{definition}

	Now we define what is called the topological pressure. In order to do that, we define the notion of a $(T,\epsilon)$ separated set: given $\epsilon,T>0$ we call a subset $S\subset X$ to be $(T,\epsilon)$ separated if and only if for all $p,q\in S$ with $p\neq q$ there exists $t\in[0,T]$ such that $d(\phi_t p,\phi_t q)>\epsilon$. Let $f:X\rightarrow\R$ be a continuous function. We define:
	\[Z(f,T,\epsilon)\defeq\sup\left\{\sum_{p\in S}\exp{\left[\int_{0}^Tf(\phi_tp)dt\right]}\mid S \text{ is } (T,\epsilon)\text{-separated}\right\}\]
	and
	\[P(\phi,f)\defeq\lim_{\epsilon\to0}\limsup_{T\to\infty}\frac{1}{T}\log Z(f,T,\epsilon).\]
	\begin{definition}
		Suppose $(X,d)$ is a compact metric space, $\phi$ is a metric Anosov flow on $X$ and suppose $f:X\rightarrow\R$ is a continuous function. Then $P(\phi,f)$ does not depend on the metric $d$ up to inducing the same topology and it is called the \emph{topological pressure} of $f$ with respect to the flow $\phi$. The topological pressure of the zero function is called the \emph{topological entropy} of the flow $\phi$ and denoted by $h_{\st}(\phi)$, i.e,  $h_{\st}(\phi)\defeq P(\phi,0)$.
	\end{definition}
	    
    \begin{remark}\label{rem:posent}
		Remark 15 on page 209 of \cite{Walters} and the discussion in the paragraph just above Theorem 5 in \cite{Pol} imply that the above limit is well defined. Moreover, Proposition 1 of \cite{Pol},  Remark 9 on page 166, Theorem 7.8 on page 174 and Corollary 7.11.1 on page 177 of \cite{Walters} imply that $h_{\st}(\phi)$ is non-negative and finite. 
	\end{remark}

	The topological pressure defined above satisfies a variational principle which plays a central role in the theory of thermodynamic formalism. Let $\cB$ be the Borel sigma algebra of $X$ and let $\mu$ be a Borel probability measure on $X$ which is invariant under $\phi$. Moreover, let $\fF$ be the set of all finite measurable partitions of $X$ and for $\cD_1,\cD_2\in\fF$, let
	\[\cD_1\vee\cD_2\defeq\{D_{1,i}\cap D_{2,j}\mid D_{1,i}\in\cD_1, D_{2,i}\in\cD_2\}.\]
	We consider the unit time flow $T\defeq\phi_1$ and for any finite measurable partition $\cD\in\fF$, denote $\cD_{(n)}\defeq\cD\vee T^{-1}\cD\vee\dots\vee T^{-n+1}\cD$. We define the measure theoretic entropy of $T$ with respect to $\mu$ as follows:
	\begin{align*}
	h(T,\mu)\defeq \sup_{\cD\in\fF}\left(\lim_{n\to\infty}\sum_{D\in\cD_{(n)}}\left[-\frac{\mu(D)\log\mu(D)}{n}\right]\right).
	\end{align*}
	\begin{remark}
		Due to Abramov's Theorem \cite{Abra} we can use the unit time flow to define the measure theoretic entropy of $\phi$ i.e. $h(\phi,\mu)\defeq h(T,\mu)$.
	\end{remark}
	\begin{theorem}[Bowen--Ruelle, Pollicott]\label{thm:bp}
		Suppose $\phi$ is a topologically transitive metric Anosov flow on the compact metric space $X$ and let $\cM^\phi$ be the space of $\phi$-invariant Borel probability measures on $X$ and suppose $f:X\rightarrow\R$ is any continuous function. Then the measure theoretic entropy and the topological pressure satisfy the following variational formula:
		\[P(\phi,f)=\sup_{\mu\in\cM^\phi}\left(h(\phi,\mu)+\int f d\mu\right).\]
	\end{theorem}
	
	\begin{remark}
		Theorems 5 and 6 of \cite{Pol} (see also Section 7.28 of \cite{ru}) imply that in Theorem \ref{thm:bp} if $f$ is H\"older continuous then there exists a unique measure $\mu_{\phi,f}$ such that
		\[P(\phi,f)=h(\phi,\mu_{\phi,f})+\int f d\mu_{\phi,f}.\]
		The measure $\mu_{\phi,f}$ is called the \emph{equilibrium state} of $f$ with respect to $\phi$. The equillibrium state of the zero function with respect to $\phi$ is called the \emph{measure of maximal entropy} of $\phi$ and is denoted by $\mu_\phi$ i.e.
		\[\mu_\phi=\mu_{\phi,0}.\]
	\end{remark}
	Please see Chapter 9 (in particular Theorems 9.7 and 9.8) of \cite{Walters} for more properties of pressure and equilibrium state.
	
	\begin{definition}
	    Let $f:X\to\R$ be a positive H\"older continuous function and for all $x\in X$, $t\in\R$, let $\beta_f(x,t)$ be such that 
	    \[\int_0^{\beta_f(x,t)}f(\phi_sx)ds=t.\]
	    Then the flow $\phi^f$ such that $\phi^f_t(x):=\phi_{\beta_f(x,t)}(x)$
	    is called the \emph{reparametrization} of the flow $\phi$ by the function $f$. Moreover, if $\phi$ is a topologically transitive metric Anosov flow and $\mu_{\phi^f}$ is the measure of maximal entropy for the reparametrized flow, then for any continuous function $g:X\to\R$, the following quantity 
	    \[I(f,g):=\int\frac{g}{f}d\mu_{\phi^f}\]
	    is called the \emph{intersection} of $f$ with $g$.
	\end{definition}
	\begin{remark}
		Similar concepts to the function $I$ defined above have been previously studied in the literature, for example by Bonahon \cite{Bona} in the context of the Teichm\"uller space and was later studied in \cite{BCLS,burg} and \cite{kni}. 
	\end{remark}
	
    We now list a few important properties of entropy, intersection and pressure which will play a central role in our article. 
    \begin{lemma}[Sambarino \cite{Samba}(Lemma 2.4)]\label{lem:samba}
        Suppose $\phi$ is a H\"older continuous flow on the compact metric space $X$ and $f:X\to\R$ is a positive H\"older continuous function. Then
	    \[P(\phi,-hf)=0 \text{ if and only if } h=h_{\st}(\phi^f).\]
    \end{lemma}
    We denote the set of periodic orbits of $\phi$ inside $X$ by $\cO$ and for any $a\in\cO$ we denote its \emph{period} by $p(a)$. Moreover, for any real valued continuous function $f$ on $X$ and for any $x\in a\in\cO$, we use the following notation:
    \[\int_a f:=\frac{1}{p(a)}\int_0^{p(a)}f(\phi_sx)ds .\]
    
    \begin{theorem}[Liv\v sic \cite{Liv}]\label{thm:liv}
        Suppose $\phi$ is a topologically transitive metric Anosov flow on a compact metric space $X$ and let $f:X\to\R$ be a H\"older continuous function. Then 
        \[\int_af=0\]
        for all $a\in\cO$ if and only if
        $f$ is Liv\v sic cohomologous to zero.
    \end{theorem}
    If  $f$ is a positive H\"older continuous function on $X$ and $K\in\R$, then
    \[R_K(f):=\left\{a\in\cO\mid p(a)\int_a f\leqslant K\right\}.\]
    We denote the cardinality of $R_K(f)$ by $|R_K(f)|$. 
    \begin{theorem}[Bridgeman--Canary--Labourie--Sambarino\cite{BCLS}]\label{thm:bcls}
        Suppose $\phi$ is a topologically transitive metric Anosov flow on a compact metric space $X$. Then for any real valued H\"older continuous function $f$ on $X$ with $f>0$ and any real valued continuous function $g$ on $X$ we have
            \[I(f,g)=\lim_{K\to\infty}\frac{1}{|R_K(f)|}\sum_{a\in R_K(f)}\frac{\int_ag}{\int_af} .\]
    \end{theorem}
    \begin{theorem}[Bowen, Ruelle and Parry--Pollicott]\label{thm:brpp}
        Suppose $\phi$ is a topologically transitive metric Anosov flow on a compact metric space $X$. Then
        \begin{enumerate}
            \item for any real valued H\"older continuous function $f$ on $X$ with $f>0$, the topological entropy of $\phi^f$ is finite and positive. Moreover, it satisfies the following identity:
            \[h_{\st}{\left(\phi^f\right)}=\lim_{K\to\infty}\frac{1}{K}\log |R_K(f)| ,\]
            \item for any real valued H\"older continuous function $g$ on $X$ we have
            \[P(\phi,g)=\lim_{K\to\infty}\frac{1}{K}\log\left(\sum_{a\in R_K(1)}e^{(p(a)\int_ag)}\right) ,\]
            \item the functions $h,P$ and $I$ vary analytically over the variables $f$ and $g$.
        \end{enumerate}
    \end{theorem}
    \begin{remark}
        The above result has been stitched together from results appearing in the following works: \cite{Bowen1,BoRu,Pol, parpol, ru}.  
    \end{remark}
    
    Suppose $f,g$ are two H\"older continuous functions such that $\int gd\mu_{\phi,f}=0$. Then the \emph{variance} of $g$ with respect to $f$ is defined as
    \[{V}(g,\mu_{\phi,f}):=\lim_{K\to\infty}\frac{1}{K}\int\left(\int_0^{K}g(\phi_sx)ds\right)^2d\mu_{\phi,f}(x).\]
    \begin{proposition}[Parry--Pollicott\cite{parpol}, Ruelle\cite{ru}] \label{prop.var}
        Suppose $f,g$ are two H\"older continuous functions on a compact metric space $X$ which admits a metric Anosov flow $\phi$. Then
        \[\left.\frac{\partial P(\phi,f+tg)}{\partial t}\right|_{t=0}=\int gd\mu_{\phi,f}.\]
        Moreover, the following holds: 
        \begin{enumerate}
            \item If $\int gd\mu_{\phi,f}=0$, then
        \[\left.\frac{\partial^2 P(\phi,f+tg)}{\partial t^2}\right|_{t=0}={V}(g,\mu_{\phi,f}).\]
        \item If ${V}(g,\mu_{\phi,f})=0$, then $g$ is Liv\v sic cohomologous to zero.
        \end{enumerate}
    \end{proposition}
    We consider the set of pressure zero H\"older continuous functions
    \[\cP_\phi(X):=\{f\mid P(\phi,f)=0\}\]
    and its tangent space at $f\in\cP_\phi(X)$, denoted by $\sT_{f}\cP_\phi(X)$. Using Proposition \ref{prop.var} we obtain that
    \[\sT_{f}\cP_\phi(X)=\left\{g\mid\int gd\mu_{\phi,f}=0\right\}.\] 
    \begin{definition}
    	Suppose $X$ is a compact metric space with a metric Anosov flow $\phi$. Suppose $\cP_\phi(X)$ is the space of pressure zero H\"older continuous functions on $X$. Then for any $f\in\cP_\phi(X)$, the \emph{pressure form} on $\sT_{f}\cP_\phi(X)$ is a bilinear map
    	$\spr_f:\sT_{f}\cP_\phi(X)\times \sT_{f}\cP_\phi(X)\to\R$ such that
    	\[\spr_f(g,g):=-\frac{{V}(g,\mu_{\phi,f})}{\int fd\mu_{\phi,f}}.\]
    \end{definition}
	\begin{remark}
		The existence of the pressure form on representation varieties was shown in \cite{BCLS}, using Propositions 4.10-4.12 of \cite{parpol}, Corollary 7.12 of \cite{ru} and Equation 1.5 (Section 1.4) of \cite{mcmu}. In fact, the pressure form made its first appearance in a similar context in \cite{mcmu}. Building upon previous works of Bowen--Series \cite{BS} and Bridgeman--Taylor \cite{BT}, McMullen showed its existence on the space of Liv\v sic cohomology classes of pressure zero H\"older continuous functions on a shift space.
	\end{remark}
    
    We end this section with the following useful result:
    \begin{proposition}[see Section 3 of \cite{BCLS}]\label{prop:bcls}
        Let $\{f_t\}_{t\in(-1,1)}$ be a one parameter family of positive H\"older continuous functions on a compact metric space $X$ which admits a metric Anosov flow $\phi$. Then
        \[h_{\st}{\left(\phi^{f_t}\right)}I(f_0,f_t)\geqslant h_{\st}{\left(\phi^{f_0}\right)}\]
        with equality if and only if $h_{\st}{\left(\phi^{f_0}\right)}f_0$ is Liv\v sic cohomologous to $h_{\st}{\left(\phi^{f_t}\right)}f_t$. Moreover, for $g=\frac{\partial f_t }{\partial t}|_{t=0}$,
        \[\spr_{f_0}(g,g)=\left.\frac{\partial^2 }{\partial t^2}\right|_{t=0} \frac{h_{\st}{\left(\phi^{f_t}\right)}}{h_{\st}{\left(\phi^{f_0}\right)}}I(f_0,f_t)\geqslant0.\]
    \end{proposition}

	\subsection{Anosov representations}\label{ss:ano}
	
	In this subsection we define Anosov representations and state some of their properties. The study of Anosov representations was introduced in \cite{Labourie} by Labourie to study Hitchin Representations in $\mathsf{SL}(n,\R)$. Later, in \cite{GW1,GW2} the notion of an Anosov representation was extended to include representations of a word hyperbolic group into any semisimple Lie group. These initial definitions of an Anosov representation were dynamical in nature and closely resembled Axiom A flows appearing in the dynamical systems literature. Characterizations of Anosov representations in terms of more algebraic criteria made their appearance in a series of works by Kapovich--Leeb--Porti \cite{KLP,KLP1,KLP2}. Subsequently, alternate proofs of more dynamical flavour of the Kapovich--Leeb--Porti characterization of Anosov representations were obtained by Bochi--Potrie--Sambarino in \cite{BPS}. In this article we will use the dynamical definition of an Anosov representation given in \cite{GW2}.
	
	In order to define Anosov representations we first need to define the Gromov flow space which plays a central role in the theory of Anosov representations. Suppose $\Gamma$ is a word hyperbolic group and let $\bdry$ be its \emph{boundary}. We consider the diagonal action of $\Gamma$ on
	\[\partial_\infty\Gamma^{(2)} := \partial_\infty\Gamma\times\partial_\infty\Gamma\setminus\{(x,x) \mid x\in \partial_\infty\Gamma\}\]
	coming from the action of $\Gamma$ on $\bdry$. Furthermore, let $\partial_\infty\Gamma^{(2)}\times\mathbb{R}$ be denoted by $\cflow$ and for all $(x,y)\in\partial_\infty\Gamma^{(2)}$ and $s,t\in\mathbb{R}$ let
	\begin{align*}
	\phi_t: \cflow&\rightarrow \cflow\\
	(x,y,s)&\mapsto (x,y,s+t).
	\end{align*}
	\begin{remark}\label{rem:gromovflow}
	In \cite{Gromov} Gromov proved that there exists a proper cocompact action of $\Gamma$ on $\cflow$ which commutes with the the flow $\{\phi_t\}_{t\in\mathbb{R}}$ and which lifts the diagonal action of $\Gamma$ on $\partial_\infty\Gamma^{(2)}$ (For more details see \cite{Champetier}, \cite{Mineyev}). Moreover, there exists a metric $d$ on $\cflow$ well defined only up to H\"older equivalence satisfying the following properties:
	\begin{enumerate}
		\item the $\Gamma$ action on $\cflow$ is isometric,
		\item the flow $\phi_t$ acts by Lipschitz homeomorphisms,
		\item every orbit of the flow $\{\phi_t\}_{t\in\mathbb{R}}$ gives a quasi-isometric embedding.
	\end{enumerate}
	The quotient space $\flow\defeq\Gamma\backslash\cflow$ is called the \textit{Gromov flow space} and by Lemma 2.3 of \cite{GT} $(\flow,d)$ is a compact connected metric space. Moreover, $\flow$ admits a partition of unity (see Section 8.2 of \cite{GT}). We use this flow space to give the dynamical definition of an Anosov representation.
	\end{remark}
	\begin{remark}
		Let $\mathsf{N}$ be a compact negatively curved manifold and suppose $\Gamma$ is the fundamental group of $\mathsf{N}$, then the flow space $\cflow$ can be seen as the unit tangent bundle of $\mathsf{N}$ endowed with the geodesic flow.
	\end{remark}
	
	\begin{definition}\label{def:an}
		Let $\sH$ be a semisimple Lie group, let $\sP_\pm$ be a pair of opposite parabolic subgroups of $\sH$, let $\sP_0\defeq\sP_+\cap\sP_-$ and suppose $\Gamma$ is a word hyperbolic group. Then an injective homomorphism $\rho:\Gamma\rightarrow \sH$ is called Anosov with respect to $\sP_\pm$ if and only if the following holds:
		\begin{enumerate}
			\item there exist continuous, injective, $\rho(\Gamma)$-equivariant maps
			\[\xi^\pm:\bdry\rightarrow\sH/\sP_\pm\]
			such that $\xi(p)\defeq(\xi^-(p_-),\xi^+(p_+))\in\sH/\sP_0\subset\sH/\sP_-\times\sH/\sP_+ $ for any $p=(p_-,p_+,t)\in\cflow$ .
			\item there exist positive real numbers $C,k$ and a continuous collection of Euclidean norms $\|\cdot\|_p$ on $\sT_{\xi(p)}\left(\sH/\sP_0\right)$ for $p\in\cflow$ such that 
			\begin{itemize}
				\item $\|\rho(\gamma)v^\pm\|_{\gamma p}=\|v^\pm\|_{p}$ for all $v^\pm\in\sT_{\xi^\pm(p_\pm)}(\sH/\sP_\pm)$ and $\gamma\in\Gamma$,
				\item $\|v^\pm\|_{\phi_{\pm t}p}\leqslant Ce^{-kt}\|v^\pm\|_p$ for all $v^\pm\in\sT_{\xi^\pm(p_\pm)}\sH/\sP_\pm$ and $t\geqslant 0$.
			\end{itemize}
		\end{enumerate}
		The map $\xi$ is called a limit map of the Anosov representation $\rho$.
	\end{definition}

    A fundamental feature of Anosov representations is that it is an open condition i.e given any $\rho\in\sHom(\Gamma,\sH)$ satisfying Definition \ref{def:an} for some opposite parabolic subgroups $\sP_\pm$, there exists a neighborhood of $\rho$ consisting only of representations satisfying Definition \ref{def:an} for the same $\sP_\pm$ (for more details see Proposition 2.1 of \cite{Labourie} and Theorem 5.13 of \cite{GW2}). Moreover, by Proposition 2.5 of \cite{GW1} Anosov representations admit a unique limit map and by Theorem 6.1 of \cite{BCLS} the limit maps vary analytically with the representation.

    \begin{remark}\label{rem:ttma}
		Suppose $\Gamma$ is a word hyperbolic group which admits an Anosov representation in $\sH$ with respect to $\sP_\pm$. Then using Proposition 4.3 of \cite{GW2} and Propositions 4.2, 5.1 of \cite{BCLS} we obtain that $\flow$ admits a topologically transitive metric Anosov flow $\psi$. The flow $\psi$ is a H\"older reparametrization of the flow $\phi$.
    \end{remark}

\subsection{Margulis spacetimes and the Margulis invariant}
	
In this subsection we state some known results in the theory of Margulis spacetimes which will be important for us to provide a context for the main results of this article and to prove them.

	\begin{convention}\label{notation1}
		Let $I_k$ denote the $k\times k$ identity matrix and let $\langle\mid\rangle$ be the quadratic form on $\R^{2n+1}$ whose corresponding matrix with respect to the standard basis $(e_1,e_2,\dots,e_{2n+1})$ has the form:
		\[
		\cQ\defeq\begin{bmatrix}
		I_{n+1} && 0\\
		0 && -I_{n}
		\end{bmatrix}
		\]
		We denote the connected component of the group of linear transformations preserving $\langle\mid\rangle$ by $\Go$. In order to simplify our notation, we henceforth denote $\Ga$ by $\sG$, $\Go$  by $\sH$, $\R^{2n+1}$ by $\sT$ and we denote their corresponding Lie algebras by $\fg,\fh,\ft$ respectively. We recall that $\sH$ is characterized by the property that it preserves orientation on the spaces spanned respectively by $(e_1,\dots,e_n)$ and $(e_{n+1},\dots,e_{2n+1})$ (for more details see Lemma 6 of \cite{neil}). We call the standard orientation on the aforementioned spaces to be positive. Hence, we can consistently endow any maximal isotropic space (dimension $n$) with positive orientation and any \emph{null space} (space orthogonal to a maximal isotropic space) with positive orientation.
	\end{convention}
	
	Now we define the notion of a Margulis spacetime as follows:
	\begin{definition}
		Suppose $\Gamma$ is a word hyperbolic group and suppose $\rho:\Gamma\rightarrow\sG$ is an injective homomorphism whose Linear part $L_\rho$ is Anosov with respect to $\sP_\pm$. Then $\rho$ is called a \emph{Margulis spacetime} if and only if $\rho(\Gamma)$ acts properly on $\R^{2n+1}$.
	\end{definition}
	\begin{remark}
		Although our definition of a Margulis spacetime includes a continuum of examples of proper affine actions \cite{AMS,GLM,GT}, yet it does not cover all known proper affine actions. Indeed, due to a result of Drumm \cite{Drumm2} the linear part of a proper affine action on $\R^3$ can contain parabolic elements. Also, more general proper affine actions have recently been found by Smilga in \cite{Smilga,Smilga3} and by Danciger--Gu\'eritaud--Kassel in \cite{DGK3}, which are not necessarily covered by our definition. 
	\end{remark}
	
	While showing the existence of proper affine actions of a non-abelian free group on $\R^3$, Margulis introduced an invariant and used it effectively to gauge properness of an affine action. Later, in \cite{AMS} this invariant was extended and defined in the case of $\sG$ to show existence of proper affine actions on $\R^{2n+1}$. We call this invariant the Margulis invariant.
	
	\begin{definition}\label{def:marginv}
	    Let $h\in\sH$ be a \emph{pseudo-hyperbolic} element (i.e. the unit eigenspace of $h$ is one dimensional and $h$ does not have $-1$ as an eigenvalue) and let
	    \[V^h_\pm:=\{v\mid \lim_{k\to\infty}h^{\mp k}v=0\}\subset \R^{2n+1}.\]
	    We note that $V^h_\pm$ are maximal isotropic subspaces and let $v^h_0$ be the unique unit eigenvector of $h$ which is positively oriented with respect to the positive orientations on $V^h_+$ and its orthogonal. Then for any $g=(h,u)\in\sG=\sH\ltimes\sT$ with $h$ pseudo-hyperbolic, the \emph{Margulis invariant} of $g$ is defined as follows:
	    \[\alpha(g):=\langle u\mid v^h_0\rangle .\]
	\end{definition}
	\begin{remark}\label{rem:marginv}
		We recall the following facts:
		\begin{enumerate}
			\item for all pseudo-hyperbolic element $h\in\sH$ we have $v_0^{hh_0h^{-1}}=hv_0^h$,
			\item for all pseudo-hyperbolic element $h\in\sH$ and $u,v\in\sT$ we have \[\alpha(h,u+v)=\alpha(h,u)+\alpha(h,v),\]
			\item for all $g_0\in\sG$ whose linear part is pseudo-hyperbolic and for all $g\in\sG$ we have 
			\begin{align*}
		        \alpha(gg_0g^{-1})&=\left\langle u + hu_0-hh_0h^{-1}u\mid v_0^{hh_0h^{-1}}\right\rangle\\
		        &=\left\langle u + hu_0-hh_0h^{-1}u\mid hv_0^{h_0}\right\rangle\\
		        &=\left\langle h^{-1}u + u_0-h_0h^{-1}u\mid v_0^{h_0}\right\rangle\\
		        &=\left\langle u_0\mid v_0^{h_0}\right\rangle=\alpha(g_0).
		    \end{align*}
		\end{enumerate}
	\end{remark}
	In the following we give a reformulation of the Margulis invariant, which is in tune with the notion of an Anosov representation.
	\begin{convention}\label{notation2}
        We fix the following two maximal isotropic subspaces:
		\[V_\pm\defeq\mathsf{Span}(\{e_i\pm e_{n+1+i}\mid i=1,\dots,n\}),\]
		and note that their orthogonals are the following null subspaces:
		\[(V_\pm)^\perp\defeq\mathsf{Span}(\{e_i\pm e_{n+1+i}\mid i=1,\dots,n\}\cup\{e_{n+1}\}),\]
		which are transverse to each other. There exist exactly two vectors of unit norm in $(V_+)^\perp \cap (V_-)^\perp$. We denote the one which is positively oriented with respect to the positive orientations on $V_+$ and its orthogonal by $v_0$. Furthermore, suppose $\sP_\pm$ respectively are the stabilizers of $(V_\pm)^\perp$ in $\sH$ with $\sP_0\defeq\sP_+\cap\sP_-$ and suppose $\sQ_\pm$ respectively are the stabilizers of $(V_\pm)^\perp$ in $\sG$ with $\sQ_0\defeq\sQ_+\cap\sQ_-$.
	\end{convention}
	Suppose $V,W$ are two transverse null spaces. We can choose $u\in V\cap W$, a basis $(v_1,\dots,v_n)$ of $V^\perp$ and for $j\in\{1,\dots,n\}$ we can choose $w_j\in W$ such that $\langle w_j\mid v_j\rangle=1$. We note that a sign change in $v_j$ results in a sign change in $w_j$ too. Then up to a sign change in $u$ and $v_1$, we obtain that there exists $h\in\sH$ such that $he_{n+1}=u$, $h(e_j+ e_{n+1+j})=v_j$ and $h(e_j- e_{n+1+j})=w_j$ for all $j\in\{1,\dots,n\}$. In particular, $h(V_+)^\perp=V$ and $h(V_-)^\perp=W$, i.e. the diagonal action of $\sH$ on the space of tuples of transverse null spaces is transitive. Hence, we get an identification of $\sH/\sP_0$ as a subspace of $(\sH/\sP_-)\times(\sH/\sP_+)$.
	
	\begin{definition}\label{def:nu}
		Suppose  $\sH$, $\sP_0$ and $v_0$ are as in Convention \ref{notation2}. We define the \emph{neutral map} as
		\begin{align*}
		\nu: \sH/\sP_0&\longrightarrow\R^{2n+1}\\
		g\sP_0&\longmapsto gv_0,
		\end{align*}
		and note that this is well defined due to Lemma 4.1 of \cite{GT}. 
	\end{definition}
	\begin{remark}
	    Suppose $\rho:\Gamma\rightarrow\sG$ is an injective homomorphism whose linear part $L_\rho$ is Anosov with respect to $\sP_\pm$. Suppose $\xi_{L_\rho}:\cflow\rightarrow\sH/\sP_0$ is the limit map of $L_\rho$ and suppose $\nu$ is the neutral map. We define $\nu_\rho\defeq\nu\circ\xi_{L_\rho}$ and observe that $\nu_\rho$ is invariant under the flow $\phi$. Hence, by abuse of notation we get a map \[\nu_\rho:\bdry^{(2)}\rightarrow\R^{2n+1}.\] 
	    We recall that $\gamma_\pm\in\bdry$ are respectively the attracting and repelling fixed points of the action of $\gamma$ on $\bdry$ and observe that 
	    \[\left(\xi^\pm_{L_\rho}(\gamma_\pm)\right)^\perp=\left\{v\mid \lim_{k\to\infty}L_\rho(\gamma^{\mp k})v=0\right\}.\]
	    Moreover, $\nu_\rho(\gamma_-,\gamma_+)$ is the unique eigenvector with unit norm and eigenvalue $1$ such that $\nu_\rho(\gamma_-,\gamma_+)$ is positively oriented with respect to the positive orientations on $(\xi^+_{L_\rho}(\gamma_+))^\perp$ and $\xi^+_{L_\rho}(\gamma_+)$. Hence, 
	    \[\nu_\rho(\gamma_-,\gamma^+) = v^{L_\rho(\gamma)}_0 \text{ and } \left(\xi^\pm_{L_\rho}(\gamma_\pm)\right)^\perp=V^{L_\rho(\gamma)}_\pm .\]
	    It follows that the Margulis invariant of $\rho(\gamma)$ for any $\gamma\in\Gamma\setminus\{e\}$ is 
	    \[\alpha(\rho(\gamma))=\langle u_\rho(\gamma)\mid\nu_\rho(\gamma_-,\gamma_+)\rangle .\]
	\end{remark}
	
	In Section \ref{sec:ir} we will show that the Margulis invariants are algebraic and using this fact, we will show that the Margulis invariant spectrum is infinitesimally rigid. 
	
	\begin{remark}
		Henceforth, we will only consider Margulis spacetimes whose linear parts are Zariski dense in $\sH$. The Zariski density assumption will play a crucial role in the proof of infinitesimal rigidity. The existence of such Margulis spacetimes are guaranteed by \cite{AMS}.
	\end{remark}

	\subsection{The moduli space and the Labourie--Margulis invariant}
	
	In this subsection we will define the moduli space of Margulis spacetimes and state some of its properties. 
	\begin{remark}\label{rem:repvar}
		Suppose $\Gamma$ is a hyperbolic group, let $\sHom_{\sA}(\Gamma,\sH)$ be the space of all injective homomorphisms of $\Gamma$ into $\sH$ which are Zariski dense, smooth points of the representation variety and Anosov with respect to $\sP_\pm$. We also recall that $\sT=\R^{2n+1}$. Moreover, for $\varrho\in\sHom_{\sA}(\Gamma,\sH)$ let us consider the following spaces:
		\begin{enumerate}
			\item the space of $\varrho$-\emph{cocycles}
			\[\mathsf{Z}^1_{\varrho}(\Gamma,\sT):=\{u:\Gamma\to\sT\mid u(\gamma\eta)=\varrho(\gamma)u(\eta)+u(\gamma)\},\]
			\item the space of $\varrho$-\emph{coboundaries}
			\[\mathsf{B}^1_{\varrho}(\Gamma,\sT):=\{v-\varrho(\cdot)v\mid v\in\sT\}\subset\mathsf{Z}^1_{\varrho}(\Gamma,\sT),\]
			\item the space of group cohomology
			\[\mathsf{H}^1_{\varrho}(\Gamma,\sT):=\mathsf{Z}^1_{\varrho}(\Gamma,\sT)/\mathsf{B}^1_{\varrho}(\Gamma,\sT).\]
		\end{enumerate}
		Let $\sHom_{\sA}(\Gamma,\sG)\subset\sHom(\Gamma,\sG)$ be the space of all representations $\rho$ whose linear parts $L_\rho\in\sHom_{\sA}(\Gamma,\sH)$. We observe that $\sHom_{\sA}(\Gamma,\sG)$ is a vector bundle over $\sHom_{\sA}(\Gamma,\sH)$ and the fiber over $\varrho\in\sHom_{\sA}(\Gamma,\sH)$ is the vector space $\mathsf{Z}^1_{\varrho}(\Gamma,\sT)$. As $\sHom_{\sA}(\Gamma,\sH)$ is open inside $\sHom(\Gamma,\sH)$, we obtain that $\sHom_{\sA}(\Gamma,\sG)$ is open inside $\sHom(\Gamma,\sG)$. We consider the conjugation action of $\sG$ on $\sHom_{\sA}(\Gamma,\sG)$ and observe that it induces a conjugation action of $\sH$ on $\sHom_{\sA}(\Gamma,\sH)$. We define \[\sA(\Gamma,\sG)\defeq\sHom_{\sA}(\Gamma,\sG)/\sG \text{ and } \sA(\Gamma,\sH)\defeq\sHom_{\sA}(\Gamma,\sH)/\sH.\] 
		It follows that $\sA(\Gamma,\sG)$ again fibers over $\sA(\Gamma,\sH)$ and the fiber of $\sA(\Gamma,\sG)$ over $[\varrho]\in\sA(\Gamma,\sH)$ can be identified with $\mathsf{H}^1_{\varrho}(\Gamma,\sT)$.
		
		Moreover, let $\sHom_{\sM}(\Gamma,\sG)\subset\sHom_{\sA}(\Gamma,\sG)$ be the space of all Margulis spacetimes. Suppose $\rho\in\sHom_{\sA}(\Gamma,\sG)$ and $g\in\sG$. We observe that $\rho(\Gamma)$ acts properly on $\sT$ if and only if $g\rho(\Gamma) g^{-1}$ acts properly on $\sT$ i.e. $\sHom_{\sM}(\Gamma,\sG)$ is invariant under the conjugation action of $\sG$. We define \[\sM(\Gamma,\sG)\defeq\sHom_{\sM}(\Gamma,\sG)/\sG.\] 
	\end{remark}
	
	 Now we will give an alternative characterization of $\sHom_{\sM}(\Gamma,\sG)$ to study its properties. In order to do that we will use the notion of a Labourie--Margulis invariant. This invariant is a Liv\v sic cohomology class of H\"older continuous functions.

	\begin{definition}
		Suppose $\rho\in\sHom_{\sA}(\Gamma,\sG)$, suppose $\phi$ is the flow on $\flow$ as mentioned in Subsection \ref{ss:ano} and let $t(\gamma)$ be the period of the periodic orbit of $\phi$ in $\flow$ corresponding to $\gamma\in\Gamma$. Then the Liv\v sic cohomology class $[f_\rho]$ of a H\"older continuous function $f_\rho:\flow\rightarrow\R$ is called the \emph{Labourie--Margulis invariant} if and only if for all non-torsion element $\gamma\in\Gamma$ the following holds:
		\[\int f_\rho d\mu_\gamma =\frac{\alpha_\rho(\gamma)}{t(\gamma)},\]
		where $\mu_\gamma$ is the $\phi$-invariant probability measure supported on the periodic orbit corresponding to $\gamma$.
	\end{definition}
	\begin{remark}
		We note that the existence of the Labourie--Margulis invariant is guaranteed by Lemma 7.2 of Ghosh--Treib \cite{GT} and its uniqueness is guaranteed by Liv\v sic's Theorem \cite{Liv}.
	\end{remark}
	Goldman--Labourie--Margulis \cite{GLM} used this invariant and showed the following result:
	\begin{theorem}
		Suppose $\Gamma$ is a free group and suppose $\rho:\Gamma\rightarrow\sG$ is an injective homomorphism such that $L_\rho$ is Fuchsian. Then $\rho$ gives rise to a Margulis spacetime if and only if its Labourie--Margulis invariant is either positive or negative.
	\end{theorem}
	This result was later generalized by Ghosh--Treib \cite{GT} for general word hyperbolic groups $\Gamma$ and $\rho\in\sHom_{\sA}(\Gamma,\sG)$.
	\begin{theorem}\label{thm:affano}
		Suppose $\Gamma$ is a word hyperbolic group and suppose $\rho:\Gamma\rightarrow\sG$ is an injective homomorphism such that $L_\rho\in\sHom_{\sA}(\Gamma,\sH)$. Then $\rho$ gives rise to a Margulis spacetime if and only if its Labourie--Margulis invariant is either positive or negative.
	\end{theorem}

	\subsection{Regularity of Labourie--Margulis invariants}
	
	In this subsection we show that the Labourie-Margulis invariant of a Margulis spacetime varies analytically over the moduli space. Later, we will use this fact to define a quadratic form on the moduli space.
	
	Suppose $\Gamma$ is a word hyperbolic group, $\flow$ is the Gromov flow space and $\psi$ is as in Remark \ref{rem:ttma}. Suppose $\pi:\cflow\rightarrow\flow$ is the natural projection map and
	\[\sigma:\sHom_{\sA}(\Gamma,\sG)\times\cflow\rightarrow\R^{2n+1}\] 
	is a map which is smooth along flow lines of $\{\psi_s\}_{s\in\R}$. We denote 
	\[(\nabla_\psi \sigma)(\rho,p)\defeq\left.\frac{d}{ds}\right|_{s=0}\sigma(\rho,\psi_s p).\]
	\begin{lemma}\label{lem:analytic}
		Suppose $\Gamma$ is a word hyperbolic group. Then there exists
		\[\sigma:\sHom_{\sA}(\Gamma,\sG)\times\cflow\rightarrow\R^{2n+1}\] 
		such that the following holds:
		\begin{enumerate}
			\item $\sigma$ is smooth along the flow lines of $\{\psi_s\}_{s\in\R}$,
			\item $\sigma$ and $\nabla_\psi \sigma$ are H\"older continuous in the variable $p\in\cflow$,
			\item $\sigma$ and $\nabla_\psi \sigma$ are analytic along the variable $\rho$,
			\item $\sigma$ and $\nabla_\psi \sigma$ are equivariant i.e. for all $\gamma\in\Gamma$ and $p\in\cflow$,
			\[\sigma(\rho,\gamma p)=\rho(\gamma) \sigma(\rho,p),\]
			\[\nabla_\psi\sigma(\rho,\gamma p)=L_\rho(\gamma) \nabla_\psi\sigma(\rho,p).\]
		\end{enumerate}
	\end{lemma}
	\begin{proof}
		We recall that any point $p\in\cflow$ is of the form $p=(p_-,p_+,t_p)$ for some unique $p_\pm\in\partial_\infty\Gamma$ and $t_p\in\R$. Let $d$ be a $\Gamma$ invariant metric on $\cflow$ such that $(\flow,d)$ is a Gromov flow space. We consider,
		\begin{align*}
		V^-(p,\epsilon)&:=\{q\mid d((q_-,p_+,t_p),p)<\epsilon\},\\
		V^+(p,\epsilon)&:=\{q\mid d((p_-,q_+,t_p),p)<\epsilon\},\\
		V^0(p,\epsilon)&:=\{q\mid d((p_-,p_+,t_q),p)<\epsilon\},
		\end{align*}
		and define $V(p,\epsilon):=V^-(p,\epsilon)\cap V^+(p,\epsilon)\cap V^0(p,\epsilon)$. We call such a neighborhood of $p$ an \textit{open cube}. As $\flow$ is compact, there exist small open cubes $\{V_i\}_{i=1}^k$ such that $\cup_{i=1}^k\pi(V_i)=\flow$. Hence $\cup_{\gamma\in\Gamma}\cup_{i=1}^k\gamma V_i=\cflow$. We know from Section 8.2 of \cite{GT} that there exist maps 
		\[\{f_i:\flow\rightarrow\R^+\}_{i=1}^k\] 
		with $\mathsf{Supp}(f_i)\subset \pi(V_i)$ such that the functions $f_i$ are H\"older continuous and smooth along flow lines with $\sum_{i=1}^kf_i=1$ and their derivatives along flow lines are also H\"older continuous. 
		We use this to construct the map 
		\[\sigma:\sHom_{\sA}(\Gamma,\sG)\times\cflow\rightarrow\R^{2n+1}.\]
		Let us fix $v\in\R^{2n+1}$. We observe that for any $p\in\Gamma V_i$ there exists a unique $\gamma_{p,i}$ such that $\gamma_{p,i}p\in V_i$. Note that in such a situation $\gamma_{\eta p,i}\eta=\gamma_{p,i}$. For $p\in\Gamma V_i$ we define $\sigma_i(\rho,p)\defeq\rho(\gamma_{p,i})^{-1}v$.
		
		Now for any $p\in\cflow$ we define:
		\[\sigma(\rho,p)\defeq\sum_{i=1}^kf_i(\pi(p))\sigma_i(\rho,p).\]
		We check that these maps are equivariant. Indeed, as
		\begin{align*}
		\sigma_\rho(\gamma p)&=\sum_{i=1}^kf_i(\pi(\gamma p))\sigma_{i,\rho}(\gamma p)=\sum_{i=1}^kf_i(\pi(p))\rho(\gamma_{\gamma p,i})^{-1}v\\
		&=\sum_{i=1}^kf_i(\pi(p))\rho(\gamma_{p,i}\gamma^{-1})^{-1}v\\
		&=\sum_{i=1}^kf_i(\pi(p))\rho(\gamma)\rho(\gamma_{p,i})^{-1}v=\sum_{i=1}^kf_i(\pi(p))\rho(\gamma)\sigma_{i,\rho}(p)\\
		&=\rho(\gamma)\sum_{i=1}^kf_i(\pi(p))\sigma_{i,\rho}(p)=\rho(\gamma)\sigma_\rho(p).
		\end{align*}
		As $f_i$ is H\"older continuous and smooth along flow lines for all $i\in\{1,\dots,k\}$, it follows from our construction that the map $\sigma$ is H\"older continuous in the variable $p$ and analytic in the variable $\rho$.
		
		Moreover, we observe that for any $p\in\cflow$ there exists $s_p>0$ such that for all real number $s$ with $|s|\leq s_p$ we have: $p\in\gamma V_i$ if and only if $\psi_sp\in\gamma V_i$. Hence by uniqueness we get $\gamma_{p,i}=\gamma_{\psi_sp,i}$ for all $s$ with $|s|\leq s_p$. We recall that $f_i$ were chosen so that the derivative of $f_i$ along flow lines is also H\"older continuous. Also, the derivative of $f_i$ along flow lines is zero outside $V_i$ for all $i\in\{1,\dots,k\}$. Hence, we obtain
		\begin{align*}
		(\nabla_\psi \sigma)(\rho,p)&=\left.\frac{d}{ds}\right|_{s=0}\sum_{i=1}^kf_i(\pi(\psi_s p))\sigma_{\rho,i}(\psi_s p)\\
		&=\left.\frac{d}{ds}\right|_{s=0}\sum_{i=1}^kf_i(\pi(\psi_s p))\rho(\gamma_{\psi_sp,i})^{-1}v\\
		&=\sum_{i=1}^k\left(\left.\frac{d}{ds}\right|_{s=0}f_i(\pi(\psi_s p))\right)\rho(\gamma_{p,i})^{-1}v,
		\end{align*}
		and it follows that $\nabla_\psi \sigma$ is analytic in the variable $\rho$ and H\"older continuous in the variable $p$. Finally, we observe that:
		\begin{align*}
		(\nabla_\psi \sigma)(\rho, \gamma p)&=\left.\frac{d}{ds}\right|_{s=0}\sigma(\rho,\psi_s \gamma p)=\left.\frac{d}{ds}\right|_{s=0}\sigma(\rho,\gamma \psi_sp)\\
		&=\left.\frac{d}{ds}\right|_{s=0}\rho(\gamma)\sigma(\rho,\psi_sp)=L_\rho(\gamma)\left.\frac{d}{ds}\right|_{s=0}\sigma(\rho,\psi_sp)\\
		&=L_\rho(\gamma)(\nabla_\psi \sigma)(\rho, p)
		\end{align*}
		and our result follows.
	\end{proof}
	
	We recall that $\cO$ denote the set of periodic orbits of $\psi$ inside $\flow$. Furthermore, we denote the period of the periodic orbit of $\psi$ inside $\flow$ corresponding to $\gamma$ by $\ell(\gamma)$. We establish the following  Corollary which shows that the Labourie-Margulis invariant defined in \cite{GT} varies analytically with respect to the representation.

	\begin{corollary}\label{cor:normalizedalpha}
		Suppose $\rho\in\sHom_{\sA}(\Gamma,\sG)$. Then there exists a H\"older continuous function $f_\rho:\flow\rightarrow\R$ such that $f_\rho$ varies analytically in the variable $\rho$ and satisfies the following equation for all $\gamma\in\cO$:
		\[\int_\gamma f_\rho=\frac{\alpha_\rho(\gamma)}{\ell(\gamma)}.\]		
	\end{corollary}
	\begin{proof}
		Suppose $\sigma_\rho$ is as in Lemma \ref{lem:analytic} and let $f_\rho:\flow\to\R$ be such that
		\[f_\rho(p)\defeq\langle\nabla_\psi\sigma_\rho(p)\mid\nu_\rho(p)\rangle\] 
		for all $p\in\flow$. We fix a point $p_\gamma\defeq(\gamma_-,\gamma_+,0)\in\cflow$ and observe that $\psi_{\ell(\gamma)}p_\gamma=\gamma p_\gamma$. Moreover, we have $\sigma_\rho(\gamma p_\gamma)=\rho(\gamma)\sigma_\rho(p_\gamma)$ and 
		\[\sigma_\rho(\psi_{\ell(\gamma)}p_\gamma)-\sigma_\rho(p_\gamma)=\int_{0}^{\ell(\gamma)}d\sigma_\rho(\psi_tp_\gamma)=\int_{0}^{\ell(\gamma)}\nabla_\psi\sigma_\rho(\psi_tp_\gamma)dt.\] 
		Hence, we compute and deduce that
		\begin{align*}
		\alpha_\rho(\gamma)&=\langle u_{\rho}(\gamma)\mid\nu_\rho(\gamma_-,\gamma_+)\rangle=\langle\rho(\gamma)\sigma_\rho(p_\gamma)-\sigma_\rho(p_\gamma)\mid\nu_\rho(\gamma_-,\gamma_+)\rangle\\
		&=\int_{0}^{\ell(\gamma)}\langle\nabla_\psi\sigma_\rho(\psi_tp_\gamma)\mid\nu_\rho(\gamma_-,\gamma_+)\rangle dt=\ell(\gamma)\int_\gamma f_\rho.
		\end{align*}
		Finally, we use Theorem 6.1 of \cite{BCLS} and Lemma \ref{lem:analytic} to conclude that $f_\rho:\flow\rightarrow\R$ is a H\"older continuous function which varies analytically in the variable $\rho$.
	\end{proof}
	
	\begin{lemma}\label{lem:LMadd}
	     Suppose $\varrho\in\sHom_{\sA}(\Gamma,\sH)$ and $u,v\in\mathsf{Z}^1_\varrho(\Gamma,\sT)$. Then for any $a,b\in\mathbb{R}$ the following holds: \[[f_{(\varrho,au+bv)}]=a[f_{(\varrho,u)}]+b[f_{(\varrho,v)}].\]
	\end{lemma}
	\begin{proof}
	    The result follows from Remark \ref{rem:marginv} (2) and Theorem \ref{thm:liv}.
	\end{proof}
	
	\begin{remark}
	    Suppose $\varrho\in\sHom_{\sA}(\Gamma,\sH)$ and $u,v\in\mathsf{Z}^1_\varrho(\Gamma,\sT)$ are such that $[f_{(\varrho,u)}]=[f_{(\varrho,v)}]$. Then  $\alpha_{(\varrho,u)}(\gamma)=\alpha_{(\varrho,v)}(\gamma)$ for all $\gamma\in\Gamma$. Now using results obtained by Drumm--Goldman \cite{DG}, Charette--Drumm \cite{CD} and Kim \cite{Kim} about marked Margulis invariant spectrum rigidity, we obtain that $(\varrho,u)$ and $(\varrho,v)$ are conjugates of each other.
	\end{remark}
	
	\begin{proposition}\label{prop:convcone}
		Suppose $\varrho\in\sHom_{\sA}(\Gamma,\sH)$ is such that $\sH^1_\varrho(\Gamma,\sT)$ admits Margulis spacetimes. Then the space of Margulis spacetimes in $\sH^1_\varrho(\Gamma,\sT)$ have two connected components which are open convex cones. These components differ by a sign change and the intersection of their closures consists of representations whose Margulis invariant spectrum is identically zero.
	\end{proposition}
	\begin{proof}
		We observe that the space of positive H\"older functions is an open cone. Hence, our result follows using Lemmas \ref{lem:analytic}, \ref{lem:LMadd} and Theorem \ref{thm:affano}.
	\end{proof}

	\section{Infinitesimal rigidity}\label{sec:ir}
	
	In \cite{DG} Drumm--Goldman showed that any two Margulis spacetimes of dimension three have distinct marked Margulis invariant spectrum. Later, this result was generalized for higher dimensions by Kim in \cite{Kim}. In this section we prove an infinitesimal version of their result. We show that Margulis spacetimes are infinitesimally determined by their marked Margulis invariant spectra. For three dimensional Margulis spacetimes this result follows from \cite{Ghosh1}. The techniques used in \cite{Ghosh1} to prove this result are not suitable for higher dimensions. Hence, in this article we use a Margulis type argument to achieve our goal. We show that the square of the Margulis invariant of an element is an algebraic function of that element and use it carefully to conclude our result.

	\subsection{Algebraicity of Margulis invariants}
	
	In this subsection we show that the Margulis invariant of an element depends algebraically on the representation. This algebraic dependence will be used in the subsequent sections to prove infinitesimal rigidity of Margulis spacetimes.
	
    Suppose $\mathfrak{gl}(2n+1,\R)$ is the space of all $(2n+1)\times(2n+1)$ matrices, let $H\in\mathfrak{gl}(2n+1,\R)$ and let $\chi_H(t)$ denote the \emph{characteristic polynomial} of $H$, i.e.
    \[\chi_H(t):=\det(tI-H).\]
    We observe that the degree of the characteristic polynomial is $2n+1$ and the coefficients of the characteristic polynomial are $c_{2n+1}(H)=1$, $c_{2n}(H)=-\tr(H)$ and $c_{2n+1-k}(H)=(-1)^k\tr(\wedge^k H)$ for $k\geqslant 2$, i.e.
    \[\chi_H(t)=\sum_{k=0}^{2n+1}c_k(H)t^k.\]
   In particular, $c_0(H)=(-1)^{2n+1}\tr(\wedge^{2n+1} H)=-\det(H)$. Hence, the coefficients of $\chi_H(t)$ are symmetric polynomials obtained from the entries of $H$. Moreover, by the Cayley-Hamilton theorem (proved by Frobenius in \cite{FCH}) we know that
    \[\chi_H(H)=0.\]
    \begin{definition}\label{def:adjugate}
        Suppose  $H\in\mathfrak{gl}(2n+1,\R)$. Then the \emph{adjugate} of $H$ is defined as follows:
        \[\adj(H):=\sum_{k=1}^{2n+1}c_k(H)H^{k-1}.\]
    \end{definition}
    \begin{remark}
    Note that $H\adj(H)=\adj(H)H=\chi_H(H)-c_0(H)=\det(H)I$.
    \end{remark}
	
	\begin{lemma}\label{lem:alg}
		The square of the map $\alpha$ can be written as a rational function in $g=(h,u)$ as follows:
		\[\alpha(g)^2=\frac{\langle[\adj(I-h)]u\mid u\rangle}{\tr[\adj(I-h)]}.\]
	\end{lemma}
	\begin{proof}
	    As $h$ is pseudo-hyperbolic, $0$ is an eigenvalue of $(I-h)$ with no multiplicity. In particular, we obtain $\det(I-h)=0$ and hence
	    \[\adj(I-h)=h\adj(I-h)=\adj(I-h)h.\]
	    It follows that for any $v\in\R^{2n+1}$, $\adj(I-h)v$ is an eigenvector of $h$ with eigenvalue $1$, i.e. $\adj(I-h)v$ is a multiple of $v_0^h$ for all $v$. 
	    Moreover, for all $w\in V_\pm^h$ we have
	    \[\adj(I-h)w=\lim_{k\to\infty}\adj(I-h)h^{\mp k}w=0.\]
	    Also, by direct computation we observe that 
	    \[\adj(I-h)v_0^h=c_1(I-h)v_0^h\] 
	    where $c_1(I-h)=\tr(\wedge^{2n}(I-h))$. As $\adj(I-h)v_0^h=c_1(I-h)v_0^h$ and $\adj(I-h)w=0$ for all $w\in (V_+^h\oplus V_-^h)= (v_0^h)^\perp$, we obtain that \[c_1(I-h)=\tr(\adj(I-h)).\]
	    Furthermore, as $(u-\alpha(g)v_0^h)\in(v_0^h)^\perp$ we obtain that \[\adj(I-h)u=\tr(\adj(I-h))\alpha(g)v_0^h.\]
	    Hence, we derive that
	    \[\langle\adj(I-h)u\mid u\rangle=\tr(\adj(I-h))\alpha(g)\langle v_0^h\mid u\rangle=\tr(\adj(I-h))\alpha(g)^2\]
	    and our result follows.
	\end{proof}
	
	\begin{definition}\label{def:dermarg}
		Suppose $g_t\in\sG$ with $g_0=g$, is an analytic one parameter family of pseudo-hyperbolic elements with
		\[G=\left.\frac{d}{dt}\right|_{t=0}g_tg^{-1}\in\mathfrak{g}.\] Then we define
		\[\dot{\alpha}(g,G)\defeq\left.\frac{d}{dt}\right|_{t=0}\alpha(g_t).\]
	\end{definition}

	\begin{corollary}\label{cor:alg}
		Suppose $g_t\in\sG$ is an analytic one parameter family of pseudo-hyperbolic elements with $g_0=g$ and tangent direction $G\in\fg$ as in Definition \ref{def:dermarg}. Then the function $\dot{\alpha}^2$ is a rational function in the variable $(g,G)$.
	\end{corollary}
	\begin{proof}
        We use Lemma \ref{lem:alg} and obtain that for some real valued multivariate polynomials $p$ and $q$,
        \[\alpha(g_t)^2=\frac{p(g_t)}{q(g_t)}.\]
        Hence by taking derivative on both sides we obtain that
        \[2\alpha(g)\dot{\alpha}(g,G)=\frac{q(g)\left.\frac{d}{dt}\right|_{t=0}p(g_t)-p(g)\left.\frac{d}{dt}\right|_{t=0}q(g_t)}{q(g)^2}\]
        As $p$ and $q$ are both real valued multivariate polynomials in $g$, we obtain that $\left.\frac{d}{dt}\right|_{t=0}p(g_t)$ and $\left.\frac{d}{dt}\right|_{t=0}q(g_t)$ both are real valued multivariate polynomials in $(g,G)$. We denote them respectively by $\dot{p}(g,G)$ and $\dot{q}(g,G)$ to obtain:
        \[\alpha(g)^2\dot{\alpha}(g,G)^2=\frac{(q(g)\dot{p}(g,G)-p(g)\dot{q}(g,G))^2}{4q(g)^4}.\]
       Finally, our result follows by computing that
        \[\dot{\alpha}(g,G)^2=\frac{(q(g)\dot{p}(g,G)-p(g)\dot{q}(g,G))^2}{4p(g)q(g)^3}.\]
	\end{proof}

	\subsection{Infinitesimal rigidity of Margulis invariants: fixed linear part}
	
	In this subsection we will prove that the marked Margulis invariant spectrum determines a Margulis spacetime up to conjugacy when its linear part is kept fixed. The argument used here is like the Margulis type argument used in \cite{Kim} to show global rigidity of the marked Margulis invariant spectrum. In fact, the arguments presented in this subsection also provide an alternate proof of global rigidity of the marked Margulis invariant spectrum of elements in $\sHom_{\sA}(\Gamma,\sG)$ with fixed linear parts. We recall that $\sH=\Go$, $\sT=\R^{2n+1}$, $\sG=\Ga$ and $\fh$, $\ft$, $\fg$ are respectively their Lie algebras.
	
	\begin{remark}\label{rem:dermarg}
	    Suppose $g=(h,v)\in\sG=\sH\ltimes\sT$ and $G=(H,V)\in\fg=\fh\oplus\ft$ then we recall that
	    \[\sad(g)(G):=gGg^{-1}=(hHh^{-1},hV-hHh^{-1}v)\in\fg=\fh\oplus\ft.\]
	    Indeed, there exist $\{g_t=(h_t,v_t)\mid t\in(-1,1)\}\subset\sG=\sH\ltimes\sT$ such that 
	    \[G=\left.\frac{d}{dt}\right|_{t=0}g_tg^{-1}\]
    	and we obtain that
    	\[\left.\frac{d}{dt}\right|_{t=0}(hh_th^{-1},v+hv_t-hh_th^{-1}v)=(hHh^{-1},hV-hHh^{-1}v).\]
    	It follows that for $G^\prime=(H^\prime,V^\prime)\in\fg=\fh\oplus\ft$ we have
    	\[[G^\prime,G]=([H^\prime,H],H^\prime V-HV^\prime).\]
    	Moreover, suppose $\rho_t\in\sHom_{\sA}(\Gamma,\sG)$ is an analytic one parameter family of representations with $\rho_0=\rho$ and for any $\gamma\in\Gamma$ let
		\[{U}(\gamma)\defeq\left.\frac{d}{dt}\right|_{t=0}\rho_t(\gamma)\rho(\gamma)^{-1}.\]
		We observe that for all $\gamma_1,\gamma_2\in\Gamma$ we have 
		\begin{align*}
		{U}(\gamma_1\gamma_2)&=\left.\frac{d}{dt}\right|_{t=0}\rho_t(\gamma_1\gamma_2)\rho(\gamma_1\gamma_2)^{-1}\\
		&=\left(\left.\frac{d}{dt}\right|_{t=0}\rho_t(\gamma_1)\rho(\gamma_2)+\rho(\gamma_1)\left.\frac{d}{dt}\right|_{t=0}\rho_t(\gamma_2)\right)\rho(\gamma_1\gamma_2)^{-1}\\
		&={U}(\gamma_1)+ \sad(\rho(\gamma_1)){U}(\gamma_2).
		\end{align*}
		Hence for all $\gamma\in\Gamma$ we obtain $(\rho(\gamma),{U}(\gamma))\in\sG\ltimes_{\sad}\fg$ and it follows that $(\rho,{U})\in\sHom(\Gamma,\sG\ltimes_{\sad}\fg)$. 
	\end{remark}
	
	\begin{notation}
	Henceforth, we use the following notations to denote various projection maps. Let $p:\sG\ltimes_{\sad}\fg\rightarrow\sG$, $\pi:\sG\ltimes_{\sad}\fg\rightarrow\sH\ltimes_{\sad}\fh$, $\pi^\prime:\sG\rightarrow\sH$ and $p^\prime: \sH\ltimes_{\sad}\fh\rightarrow\sH$ be the natural projection maps. Then $\pi,\pi^\prime,p,p^\prime$ are homomorphisms and the following diagram commutes:
		\[ \begin{tikzcd}
		\sG\ltimes_{\sad}\fg \arrow{r}{p} \arrow[swap]{d}{\pi} & \sG \arrow{d}{\pi^\prime} \\%
		\sH\ltimes_{\sad}\fh \arrow{r}{p^\prime}& \sH.
		\end{tikzcd}
		\] 
	Moreover, by abusing notation we also denote the natural projection map from $\sHom(\Gamma,\sG\ltimes_{\sad}\fg)$ to $\sHom(\Gamma,\sG)$ by $p$ and define \[\sHom_{\sA}(\Gamma,\sG\ltimes_{\sad}\fg)\defeq p^{-1}\left(\sHom_{\sA}(\Gamma,\sG)\right).\]
	We also define the following spaces:
	\begin{enumerate}
	    \item the space of $\sad\circ\rho$-\emph{cocycles}
	    \[\mathsf{Z}^1_{\sad\circ\rho}(\Gamma,\fg):=\{{U}:\Gamma\to\fg\mid {U}(\gamma\eta)=\sad(\rho(\gamma)){U}(\eta)+{U}(\gamma)\},\]
	    \item the space of $\sad\circ\rho$-\emph{coboundaries}
	    \[\mathsf{B}^1_{\sad\circ\rho}(\Gamma,\fg):=\{G-\sad\circ\rho(\cdot)(G)\mid G\in\fg\}\subset\mathsf{Z}^1_{\sad\circ\rho}(\Gamma,\fg),\]
	    \item the space of group cohomologies
	    \[\mathsf{H}^1_{\sad\circ\rho}(\Gamma,\fg):=\mathsf{Z}^1_{\sad\circ\rho}(\Gamma,\fg)/\mathsf{B}^1_{\sad\circ\rho}(\Gamma,\fg).\]
	\end{enumerate}

	\end{notation}
	
	\begin{lemma}\label{lem:fixedlin}
		Suppose $(g,G)\in\sG\ltimes_{\sad}\fg$ be such that $g=(h,v)\in\sH\ltimes\sT$ and $G=(0,V)\in\fh\oplus\ft$. Then 
		\[\dot{\alpha}{(h,v,0,V)}=\alpha(h,V).\]
	\end{lemma}
	\begin{proof}
	    Let $\{g_t=(h,v+tV)\mid t\in(-1,1)\}\subset\sG=\sH\ltimes\sT$. Then
	    \[G=\left.\frac{d}{dt}\right|_{t=0}g_tg^{-1}=\left.\frac{d}{dt}\right|_{t=0}(h,v+tV)(h^{-1},-h^{-1}v)=(0,V).\]
	    Finally, we have
	    \[\dot{\alpha}(g,G)=\left.\frac{d}{dt}\right|_{t=0}\alpha(g_t)=\left.\frac{d}{dt}\right|_{t=0}\langle v+tV\mid v_0^h\rangle=\langle V\mid v_0^h\rangle=\alpha(h,V).\]
    \end{proof}	    
	Suppose $\rho\in\sHom_{\sA}(\Gamma,\sG)$ and ${U}\in\mathsf{Z}^1_{\sad\circ\rho}(\Gamma,\fg)$. Then observe that $(\rho,{U})\subset\sHom_{\sA}(\Gamma,\sG\ltimes_\sad\fg)$.
    \begin{lemma}\label{lem:der}
        Suppose ${U}\in\mathsf{Z}^1_{\sad\circ\rho}(\Gamma,\fg)$ is such that $\pi((\rho,U)(\Gamma))\subset\sH\subset\sH\ltimes_\sad\fh$ i.e. for all $\gamma\in\Gamma$ 
        \[{U}(\gamma)=(0,{V}(\gamma))\in\fg=\fh\oplus\ft\] 
        for some $V(\gamma)\in\ft$. Then $V\in\mathsf{Z}^1_{L_\rho}(\Gamma,\ft)\cong\mathsf{Z}^1_{L_\rho}(\Gamma,\sT)$.
    \end{lemma}
	\begin{proof}  
	We use Remark \ref{def:dermarg} and observe that for $\gamma,\eta\in\Gamma$ the following holds:
		\begin{align*}
		    (0,{V}(\gamma\eta))&={U}(\gamma\eta)=\rho(\gamma){U}(\eta)\rho(\gamma)^{-1}+U(\gamma)\\
		    &=(L_\rho(\gamma),u_\rho(\gamma))(0,{V}(\eta))(L_\rho(\gamma)^{-1},-L_\rho(\gamma)^{-1}u_\rho(\gamma))+(0,{V}(\gamma))\\
		    &=(0,L_\rho(\gamma){V}(\eta)-0)+(0,{V}(\gamma))=(0,L_\rho(\gamma){V}(\eta)+{V}(\gamma)).
		\end{align*}
	Hence, $V\in\mathsf{Z}^1_{L_\rho}(\Gamma,\ft)$. Moreover, we recall that $\sT=\R^{2n+1}$ and it follows that $\sT\cong\ft$.
	\end{proof}
	
	\begin{proposition}\label{prop:lin}
		Suppose $\rho:\Gamma\to\sG$ is a Margulis spacetime such that $L_\rho(\Gamma)$ is Zariski dense in $\sH$ and $U\in\mathsf{Z}^1_{\sad\circ\rho}(\Gamma,\fg)$ with $\pi((\rho,U)(\Gamma))\subset\sH$. Then $U\in\mathsf{B}^1_{\sad\circ\rho}(\Gamma,\fg)$ if and only if 
		\[\dot{\alpha}(\rho(\gamma),{U}(\gamma))=0\]
		for all $\gamma\in\Gamma$.
	\end{proposition}
	\begin{proof}
		We use Lemma \ref{lem:der} to obtain that $U=(0,V)$ for some $V\in\mathsf{Z}^1_{L_\rho}(\Gamma,\sT)$ and use Lemma \ref{lem:fixedlin} to deduce that
		\[\dot{\alpha}(\rho(\gamma),{U}(\gamma))=\alpha(L_\rho(\gamma),V(\gamma)).\]
		
		We first prove that if $\dot{\alpha}(\rho(\gamma),{U}(\gamma))=0$ for all $\gamma\in\Gamma$ then $U\in\mathsf{B}^1_{\sad\circ\rho}(\Gamma,\fg)$. We recall from Lemma \ref{lem:alg} that $\alpha^2$ is an algebraic function on $\sG$. Hence the Zariski closure $\cG$ of $(L_\rho,V)(\Gamma)$ inside $\sG=\sH\ltimes\sT$ is also contained in the zero set of $\alpha^2$.
		
		Recall that the natural projection map $\pi^\prime:\sG\rightarrow\sH$ is a homomorphism and in the following three points we show that $\pi^\prime:\cG\rightarrow\sH$ is, in fact, an isomorphism. 
		\begin{enumerate}
			\item ($\pi^\prime(\cG)=\sH$): We observe that $\cG$ is normalized by $(L_\rho,V)(\Gamma)$. Hence, $\pi^\prime(\cG)$ is normalized by $L_\rho(\Gamma)$. Also, $L_\rho(\Gamma)$ is Zariski dense in $\sH$. Hence $\pi^\prime(\cG)$ is a normal subgroup of $\sH$. But $\sH$ is simple and $L_\rho(\Gamma)$ is not trivial. Therefore, we obtain $\pi^\prime(\cG)=\sH$.
			\item ($\cG\cap\sT\subsetneq\sT$): Suppose $h\in\sH$ is a pseudo-hyperbolic element. We use the definition of a Margulis invariant and observe that \[\alpha(h,v_0^h)=\langle v_0^h\mid v_0^h\rangle=1\neq0.\] 
			Therefore, it follows that $\left(h,v_0^h\right)\notin\cG$ i.e. $\cG\subsetneq\sG$. We claim that this implies $\cG\cap\sT\subsetneq\sT$. Indeed, on the contrary, if $\cG\cap\sT=\sT$ then using point (1) we get $\cG=\sG$. Which is a contradiction.
			\item ($\cG\cap\sT$ is trivial): We recall that $\cG$ is normalized by $(L_\rho,V)(\Gamma)$. Hence for any $(I,v)\in\cG\cap\sT$ and for any $\gamma\in\Gamma$ we have
			\[(L_\rho(\gamma),V(\gamma))(I,v)(L_\rho(\gamma),V(\gamma))^{-1}=(I,L_\rho(\gamma)v)\in\cG\cap\sT.\]
			Moreover, $L_\rho(\Gamma)$ is Zariski dense in $\sH$. Hence for any $(I,v)\in\cG\cap\sT$ and any $h\in\sH$ we obtain that $(I,hv)\in\cG\cap\sT$. But the action of $\sH$ on $\sT$ does not preserve any non-trivial proper subspaces. Hence our claim follows from the above point (2).
		\end{enumerate}
		Therefore, the homomorphism $\pi^\prime:\cG\rightarrow\sH$ is an isomorphism of algebraic groups. Let $\sv(g)\in\sT$ be such that $(g,\sv(g))=(\left.\pi^\prime\right|_\cG)^{-1}(g)$ for all $g\in\sH$.
		As $\pi^\prime$ is a smooth isomorphism we get that $\sv:\sH\rightarrow\sT$ is a smooth map satisfying the following property for all $g,h\in\sH$:
		\[\sv(gh)=g\sv(h)+\sv(g).\]
		Hence, $\sv$ induces a map, which we again denote by $\sv$, from $\fh$ to $\ft$. We compute and observe that $\sv:\fh\to\ft$ is a derivation i.e. for all $G,H\in\fh$
		\[\sv([G,H])=G\sv(H)-H\sv(G)\]
		where $[,]$ denotes the Lie bracket of $\fh$. We use Whitehead's Lemma (please see end of section 1.3.1 in page 13 of \cite{Raghu}) to conclude that $\sv:\fh\to\ft$ is an inner derivation i.e. there exists some $V\in\ft$ such that
		\[\sv(H)=HV.\]
		We observe that $\sv_1:\sH\to\sT$, such that $\sv_1(h)=hV-V$, satisfy the equation:
		\[\sv_1(gh)=g\sv_1(h)+\sv_1(g).\]
		Now using uniqueness of solutions of first order differential equations we obtain that $\sv=\sv_1$.
		As $(L_\rho,V)(\Gamma)\subset\cG$, it follows that \[V(\gamma)=\sv(L_\rho(\gamma))=L_\rho(\gamma)V-V\] 
		for all $\gamma\in\Gamma$. Hence, for all $\gamma\in\Gamma$ we deduce that
		\[U(\gamma)=(0,L_\rho(\gamma)V-V)=\sad(\rho(\gamma))(0,V)-(0,V)\]
		and it follows that $U\in\mathsf{B}^1_{\sad\circ\rho}(\Gamma,\fg)$.

		Now we will prove the other implication. Suppose $U\in\mathsf{B}^1_{\sad\circ\rho}(\Gamma,\fg)$. We recall that $U=(0,V)$ for some $V\in\mathsf{Z}^1_{L_\rho}(\Gamma,\sT)$. Hence, there exists some $v\in\sT$ such that
		\[V(\gamma)=L_\rho(\gamma)v-v.\] 
		Then using Lemma \ref{lem:der} we conclude that $\dot{\alpha}{(\rho(\gamma),{U}(\gamma))}=\alpha(L_\rho(\gamma),V(\gamma))=0$ for all $\gamma\in\Gamma$.
	\end{proof}

	\subsection{Infinitesimal rigidity of Margulis invariants: general}
	Drumm--Goldman \cite{DG} and Kim \cite{Kim} showed that any two points on $\sM(\Gamma,\sG)$ have distinct marked Margulis invariant spectrums. In this subsection we prove an infinitesimal version of this result.
	\begin{lemma}\label{lem.zd}
		Suppose $\varrho\in\sHom_{\sA}(\Gamma,\sG)$ and $u\in\mathsf{Z}^1_{\varrho}(\Gamma,\sT)$ is a $\varrho$-cocycles which is not a $\varrho$-coboundary. Then $(\varrho,u)(\Gamma)$ is Zariski dense inside $\sG$.
	\end{lemma}
	\begin{proof}
		Let $\cR$ be the Zariski closure of $(\varrho,u)(\Gamma)$ inside $\sG$. As $(\varrho,u)(\Gamma)$ normalizes $\cR$, we obtain that $\varrho(\Gamma)$, a Zariski dense subgroup of $\sH$, normalizes $\pi^\prime(\cR)$. Hence $\pi^\prime(\cR)$ is a normal subgroup of $\sH$. But $\sH$ is simple and $\varrho(\Gamma)$ is not trivial. Therefore, we obtain $\pi^\prime(\cR)=\sH$.
		
		Now if $\cR\cap\sT$ is trivial, then $\pi^\prime$ gives an isomorphism between $\cR$ and $\sH$. Now we use Whitehead's Lemma to get that $\cR$ is conjugate to $\sH$. Hence $\alpha_{(\varrho,u)}(\gamma)=0$ for all $\gamma\in\Gamma$, a contradiction as $u$ is not a $\varrho$-coboundary. 
		
		Therefore, there exists $(I,v)\in\cR\cap\sT$ with $v\neq0$. Moreover, as $\cR\cap\sT$ is normal inside $\cR$ and $\pi^\prime(\cR)=\sH$, we obtain that $(I,hv)\in\cR\cap\sT$ for all $h\in\sH$. As $\sH$ does not preserve any non-trivial proper subspace of $\sT$, we conclude that $\cR\cap\sT=\sT$ and it follows that $\cR=\sG$.
	\end{proof}

	\begin{lemma}\label{lem:zerodmarg}
		Suppose $g\in\sG$, $G\in\fg$ and $g_t=(h_t,u_t)\in\sG=\sH\ltimes\sT$ is a one parameter family with $g_0=g$ and $\left.\frac{d}{dt}\right|_{t=0}g_tg^{-1}=G$. Then 
		\[\dot{\alpha}(g,\sad(g)(G)-G)=0.\]
	\end{lemma}
	\begin{proof}
	    We use Remark \ref{rem:marginv} (3) to obtain that $\alpha(g_tgg_t^{-1})=\alpha(g)$ and conclude by observing that
		\[\dot{\alpha}(g,G-\sad(g)(G))=\left.\frac{d}{dt}\right|_{t=0}\alpha(g_tgg_t^{-1})=\left.\frac{d}{dt}\right|_{t=0}\alpha(g)=0.\]
	\end{proof}
	
	\begin{theorem}\label{thm:infriggen}
		Suppose $\rho:\Gamma\to\sG$ is a Margulis spacetime with a Zariski dense linear part and ${U}\in\mathsf{Z}^1_{\sad\circ\rho}(\Gamma,\fg)$. Then ${U}\in\mathsf{B}^1_{\sad\circ\rho}(\Gamma,\fg)$ if and only if 
		\[\dot{\alpha}(\rho(\gamma),{U}(\gamma))=0\]
		for all $\gamma\in\Gamma$.
	\end{theorem}
	\begin{proof}
		If ${U}\in\mathsf{B}^1_{\sad\circ\rho}(\Gamma,\fg)$, then there exists some $G\in\fg$ such that 
		\[{U}(\gamma)=G-\sad\circ\rho(\gamma)(G)\]
		for all $\gamma\in\Gamma$. Now using Lemma \ref{lem:zerodmarg} we deduce that for all $\gamma\in\Gamma$,
		\[\dot{\alpha}(\rho(\gamma),{U}(\gamma))=\dot{\alpha}(\rho(\gamma),G-\sad\circ\rho(\gamma)(G))=0.\]

		Now we show that if $\dot{\alpha}(\rho(\gamma),{U}(\gamma))=0$ for all $\gamma\in\Gamma$, then $U\in\mathsf{B}^1_{\sad\circ\rho}(\Gamma,\fg)$.	We use Corollary \ref{cor:alg} to get that $\dot{\alpha}^2$ is an algebraic function on $\sG\ltimes_{\sad}\fg$. Hence, the Zariski closure $\cH$ of $(\rho,{U})(\Gamma)$ inside $\sG\ltimes_{\sad}\fg$ is also contained in the zero set of $\dot{\alpha}^2$.
		
		We observe that the kernel of $p: \cH\rightarrow \sG$, denoted by $\fk$, satisfies $\fk=\fg\cap\cH$. In the following four steps we show that $p:\cH\to\sG$ is an isomorphism:
		\begin{enumerate}
			\item ($p(\cH)=\sG$): Indeed, as $\rho(\Gamma)\subset p(\cH)$, $p(\cH)$ is normalized by $\rho(\Gamma)$ and $\rho(\Gamma)$ is Zariski dense in $\sG$ (see Lemma \ref{lem.zd}), we deduce that $p(\cH)$ is a normal subgroup of $\sG$. We observe that $\sT$ contains all proper normal subgroups of $\sG$ and as $\rho$ is a Margulis spacetime with a Zariski dense linear part, $\rho(\Gamma)$ is not a subgroup of $\sT$. Hence $p(\cH)$ is also not a subgroup of $\sT$ but $p(\cH)$ is normal inside $\sG$. Therefore, we conclude $p(\cH)=\sG$.
			
			\item ($\fk\subsetneq\fg$): Suppose $h\in\sH$ is a pseudo-hyperbolic element. We use Lemma \ref{lem:fixedlin} and observe that \[\dot{\alpha}(h,0,0,v_0^h)=\alpha(h,v_0^h)=\langle v_0^h\mid v_0^h\rangle=1\neq0.\] 
			As $\cH$ is a subset of the zero set of $\dot{\alpha}^2$, we obtain $\cH\subsetneq\sG\ltimes_{\sad}\fg$. We claim that this implies $\fk\subsetneq\fg$. Indeed, on the contrary, if $\fk=\fg$ then using point (1) we get $\cH=\sG\ltimes_{\sad}\fg$, a contradiction.
			
			\item ($\fk\subset\ft$): We observe that $\cH$ is normalized by $(\rho,{U})(\Gamma)$. Hence $\fk$, the kernel of $p$, is normalized by $\rho(\Gamma)$ and $\rho(\Gamma)$ is Zariski dense in $\sG$. It follows that $\fk$ is an ideal of $\fg$ i.e $[\fk,\fg]\subset\fk$. Also, $\fh$ is simple and $\fk\cap\fh$ is an ideal of $\fh$, hence $\fk\cap\fh$ is either trivial or $\fh$. Therefore, we will obtain $\fk\subset\ft$ once we show that $\fk\cap\fh\neq\fh$. We will show this via contradiction. On the contrary, suppose $\fk\cap\fh=\fh$. Using point (2) we get that $\fk=\fh+\fu$ where $\fu\subsetneq\ft$. Moreover, $[\fk,\fg]\subset\fk$ implies that $\fh\fu\subset\fu$. Indeed, for any $(H^\prime,0)\in\fg=\fh+\ft$ and $(H,\tau)\in\fk$ we have \[[(H^\prime,0),(H,\tau)]=([H^\prime,H],H^\prime\tau)\in\fk=\fh+\fu.\] 
			We also know that $\sH$ does not preserve any non-trivial proper subspace of $\sT$. Hence, $\fu$ is trivial and $\fk=\fh$. This is a contradiction since $\fk$ is an ideal of $\fg$ but $\fh$ is not an ideal of $\fg$.
			
			\item ($\fk$ is trivial): We observe that $\ft$ is an ideal of $\fg$. We also know that $\sH$ does not preserve any non-trivial proper subspace of $\sT$. Hence, either $\fk$ is trivial or $\fk=\ft$. If possible, let us assume that $\fk=\ft$. Then $p$ induces an isomorphism between the two algebraic groups
			$\cH/\ft$ and $\sG$. We again denote this isomorphism by $p$. We also note that
			\[\cH/\ft\subset(\sG\ltimes_\sad\fg)/\ft\cong\sG\ltimes_\sad(\fg/\ft).\]
			Let $\sv(g)\in(\fg/\ft)$ be such that \[\left(p|_{(\cH/\ft)}\right)^{-1}(g)=(g,\sv(g))\in(\cH/\ft)\subset\sG\ltimes_\sad(\fg/\ft).\]
			It follows that $\sv:\sG\to(\fg/\ft)$ is a smooth map satisfying
			\[\sv(gh)=\sad(g)\sv(h)+\sv(g)\]
			for all $g,h\in\sG$. Hence, $\sv$ induces a map, which we again denote by $\sv$, from $\fg$ to $\fg/\ft$. We observe that for all $G,H\in\fg$,
			\[\sv([G,H])=[G,\sv(H)]-[H,\sv(G)].\]
			Hence, $\sv$ is a derivation and by Whitehead's Lemma (please see end of section 1.3.1 in page 13 of \cite{Raghu}) we obtain that $\sv$ is an inner derivation i.e. there exist some $\Psi\in(\fg/\ft)$ such that
			\[\sv(G)=[G,\Psi]\]
			for all $G\in\fg$.	We consider $\sv_1:\sG\to(\fg/\ft)$ such that for all $g\in\sG$,
			\[\sv_1(g):=\sad(g)(\Psi)-\Psi\]
			and observe that for all $g,h\in\sG$ we have			
			\[\sv_1(gh)=\sad(g)\sv_1(h)+\sv_1(g).\]
			Now by uniqueness of the solutions of first order differential equations we obtain that $\sv=\sv_1$.
			It follows that
			\[{U}(\gamma)+\ft=\sad(\rho(\gamma))(\Psi)-\Psi.\]
			As $\Psi\in(\fg/\ft)$, there exists some $H_0\in\fh$ such that $G_0=(H_0,0)\in\fg=\fh\oplus\ft$ and $\Psi=G_0+\ft$. Therefore, we deduce that
			\begin{align*}
			{U}(\gamma)+\ft&=\sad(\rho(\gamma))(G_0+\ft)-(G_0+\ft)\\
			&=(\sad\circ L_\rho(\gamma)(H_0)-H_0,0)+\ft
			\end{align*}
			We consider the conjugate representation \[(\rho,W):=(e,G_0)(\rho,{U})(e,-G_0)\] 
			and observe that
			$(\rho,W)(\Gamma)\subset\sG\ltimes_\sad(\{0\}\oplus\ft)$. Hence,
			\[(e,G_0)\cH(e,-G_0)\subset\sG\ltimes_\sad(\{0\}\oplus\ft).\]
			Moreover, $p(\cH)=\sG$ and $\fk=\ft$ imply that 
			\[(e,G_0)\cH(e,-G_0)=\sG\ltimes_\sad(\{0\}\oplus\ft).\]
			Suppose $h$ is a pseudo-hyperbolic element of $\sH$, $g_1=(h,0)\in\sG$ and $G_1=(0,v_0^h)\in\{0\}\oplus\ft$.
			Then $(g_1,G_1)\in\sG\ltimes_\sad(\{0\}\oplus\ft)$ and \[\dot{\alpha}(g_1,G_1)\neq0.\] 
			This is a contradiction since $\dot{\alpha}(\cH)\subset\{0\}$, 			 $(e,-G_0)(g_1,G_1)(e,G_0))=(g_1,G_1-G_0+\sad(g_1)(G_0))\in\cH$ and by Lemma \ref{lem:zerodmarg}
			\[\dot{\alpha}(g_1,G_1-G_0+\sad(g_1)(G_0))=\dot{\alpha}(g_1,G_1)\neq0.\]
			Therefore, our assumption was wrong and we deduce that $\ft$ is trivial.		
		\end{enumerate}
	
		As $p$ is an isomorphism and $\cH\subset\sG\ltimes_\sad\fg$ we obtain that 
		\[(p|_\cH)^{-1}(g)=(g,\sw(g))\in\sG\ltimes_\sad\fg\] 
		for some $\sw:\sG\to\fg$ and all $g\in\sG$. Furthermore, we check that
		\[\sw(gh)=\sad(g)\sw(h)+\sw(g)\]
		for all $g,h\in\sG$. Arguing similarly as before, using Whitehead's lemma, we obtain that there exists some $G_0\in\fg$ such that for all $g\in\sG$,
		\[\sw(g)=\sad(g)(G_0)-G_0.\]
		Therefore, for all $\gamma\in\Gamma$ we deduce that
		\[{U}(\gamma)=\sad(\rho(\gamma))(G_0)-G_0\]
		i.e. ${U}\in\mathsf{B}^1_{\sad\circ\rho}(\Gamma,\fg)$ and our result follows.		
	\end{proof}

	\section{Pressure form and convexity} \label{sec:pfc}
	
	In this section we define the pressure form on the moduli space of Margulis spacetimes and show that the restrictions of the pressure form on constant entropy sections of the moduli space are Riemannian. Moreover, we also show that the space of constant entropy sections with fixed linear parts is the boundary of a convex domain.
	
	\subsection{Pressure form on the moduli space}
	
	In this subsection we will define and study the pressure form on the moduli space of Margulis spacetimes. We will pullback concepts from thermodynamic formalism via the Labourie-Margulis invariant. Suppose $\Gamma$ is a word hyperbolic group which admits a Margulis spacetime. Then $\Gamma$ admits an Anosov representation. We know from Remark \ref{rem:ttma} that if a word hyperbolic group admits an Anosov representation then $\flow$ admits a topologically transitive metric Anosov flow. Suppose $\psi$ is the topologically transitive metric Anosov flow on $\flow$ as mentioned in Remark \ref{rem:ttma} and $\rho\in\sHom_{\sM}(\Gamma,\sG)$, $\rho_1\in\sHom_{\sA}(\Gamma,\sG)$. We define $R_T(\rho)\defeq R_T(f_\rho)=\{\gamma\in\cO\mid\alpha_\rho(\gamma)\leqslant T\}$ and consider the following:
	\begin{align*}
	h_\rho&\defeq h_{\st}{\left(\phi^{f_\rho}\right)}=\lim_{T\to\infty}\frac{1}{T}\log|R_T(\rho)|,\\
	I(\rho,\rho_1)&\defeq I(f_\rho,f_{\rho_1})=\lim_{T\to\infty}\frac{1}{|R_T(\rho)|}\sum_{\gamma\in R_T(\rho)}\frac{\alpha_{\rho_1}(\gamma)}{\alpha_\rho(\gamma)}.
	\end{align*}
	
	\begin{lemma}\label{lem:analent}
		Suppose $\rho\in\sHom_{\sM}(\Gamma,\sG)$ and $\rho_1\in\sHom_{\sA}(\Gamma,\sG)$. Then both $h_\rho$ and $I(\rho,\rho_1)$ exist and are finite. Moreover, $h_\rho$ is positive and the following maps are analytic:
		\begin{align*}
		h:\sHom_{\sM}(\Gamma,\sG)&\longrightarrow\R\\
		\rho&\longmapsto h_\rho,\\
		I:\sHom_{\sM}(\Gamma,\sG)\times\sHom_{\sA}(\Gamma,\sG)&\longrightarrow\R\\
		(\rho,\rho_1)&\longmapsto I(\rho,\rho_1).
		\end{align*}
	\end{lemma}
	\begin{proof}
		We note that ($\flow$, $\psi$) is topologically transitive and metric Anosov. Hence by Remark \ref{rem:posent} $h_{\st}\psi$ is positive. Now by Lemma \ref{lem:analytic} and Corollary \ref{cor:normalizedalpha} there exist H\"older continuous functions $f_\rho$ varying analytically in the coordinate $\rho$ such that $\int_\gamma f_\rho=\frac{\alpha_\rho(\gamma)}{\ell(\gamma)}$ for all $\gamma\in\cO$. Moreover, by the main theorem of \cite{GT} we get that for $\rho\in\sHom_{\sM}(\Gamma,\sG)$ the functions $f_\rho$ are Liv\v sic cohomologous to positive functions on the compact set $\flow$. Hence, for $\rho\in\sHom_{\sM}(\Gamma,\sG)$ there exist positive constants $c_\rho,C_\rho$ such that for all $\gamma\in\cO$:
		\[0<c_\rho<\frac{\alpha_\rho(\gamma)}{\ell(\gamma)}<C_\rho.\]
		Therefore, we obtain $R_{\frac{T}{C_\rho}}(\psi)\subset R_T(\rho)\subset R_{\frac{T}{c_\rho}}(\psi)$ and it follows that 
		\[0<\frac{h_{\st}\psi}{C_\rho}\leqslant h_\rho \leqslant \frac{h_{\st}\psi}{c_\rho}.\]
		Suppose $f_\rho$ is Liv\v sic cohomologous to $g_\rho>0$. Then the function $\frac{f_{\rho_1}}{g_\rho}$ is well defined and is a continuous function on the compact set $\flow$. Hence, there exist real numbers $c,C$, not necessarily positive, such that $c<\frac{f_{\rho_1}}{g_\rho}<C$. Therefore, 
		\[c\alpha_{\rho}(\gamma)=c\ell(\gamma)\int_\gamma g_{\rho}<\alpha_{\rho_1}(\gamma)=\ell(\gamma)\int_\gamma f_{\rho_1}<C\ell(\gamma)\int_\gamma g_{\rho}=C\alpha_{\rho}(\gamma)\]
		and it follows that $I(\rho,\rho_1)$ is finite and can possibly be non-positive. Finally, using Theorem \ref{thm:brpp} we get that the functions $h$ and $I$ are analytic.
	\end{proof}
	\begin{remark}
		Similar results about analytic variations of the topological entropy were obtained for Anosov flows on closed Riemannian manifolds in \cite{KKPW} and for projective Anosov representations in \cite{BCLS}.
	\end{remark}
	Suppose $\rho\in\sHom_{\sM}(\Gamma,\sG)$. Then by \cite{GT} we know that $f_\rho$ is Liv\v sic cohomologous to a positive function. Hence, we can use the variational principle and Abramov's formula \cite{Abra} to obtain that $P(\psi,-hf_\rho)=0$ if and only if $h=h_\rho$ (see Lemma \ref{lem:samba}). Let $\cH(\flow)$ be the space of Liv\v sic cohomology classes of pressure zero H\"older continuous functions on $\flow$. We define the thermodynamic mapping as follows:
	\begin{align*}
		\sth: \sHom_{\sM}(\Gamma,\sG) &\longrightarrow \cH(\flow)\\
		\rho &\longmapsto [-h_\rho f_\rho].
	\end{align*}
	We use Remark \ref{rem:marginv} (3), Corollary \ref{cor:normalizedalpha} and Theorem \ref{thm:liv} to deduce that the map $\sth$ induces the following map which we again denote by $\sth$:
	\begin{align*}
		\sth: \sM(\Gamma,\sG) &\longrightarrow \cH(\flow)\\
		[\rho] &\longmapsto [-h_\rho f_\rho].
	\end{align*}
	\begin{remark}\label{rem:prM}
		We recall that $\flow$ is a compact metric space, $\psi$ is a topologically transitive metric Anosov flow on $\flow$ and for $\rho\in\sHom_{\sM}(\Gamma,\sG)$ the functions $f_\rho$ are Liv\v sic cohomologous to positive functions. Moreover, entropy, intersection and pressure depend only on the Liv\v sic cohomology class of a function. Hence, the space $\cH(\flow)$ admits a pressure form coming from the pressure form on the space $\cP_\psi(\flow)$. In this article, we consider the pull-back of the pressure form on $\cH(\flow)$ via $\sth$ to define a pressure form on $\sM(\Gamma,\sG)$:
		\[\spr:\sT{\sM}(\Gamma,\sG)\times\sT{\sM}(\Gamma,\sG)\rightarrow\R.\] 
		Finally, using Proposition \ref{prop:bcls} we obtain that
		the pressure form on $\sM(\Gamma,\sG)$ is a positive semi-definite bilinear form. 
	\end{remark}
	\begin{remark}
		We recall that $\sHom_{\sA}(\Gamma,\sH)$ consists of Zariski dense, smooth points of the representation variety $\sHom(\Gamma,\sH)$ which are also Anosov with respect to suitable parabolic subgroups. Hence, the quotient space $\sA(\Gamma,\sH)$ is a manifold (for more details please see \cite{LuMa,JM,BCLS}). We use Remark \ref{rem:repvar} to recall that $\sA(\Gamma,\sG)$ is a vector bundle over $\sA(\Gamma,\sH)$. Hence, $\sA(\Gamma,\sG)$ is also a manifold.  Moreover, using Corollary \ref{cor:normalizedalpha} and Theorem \ref{thm:affano} to deduce that  and $\sM(\Gamma,\sG)$ is an open subset of $\sA(\Gamma,\sG)$. Hence, the space $\sM(\Gamma,\sG)$ is also a manifold and for $\rho\in\sHom_{\sM}(\Gamma,\sG)$ we have
		\begin{align*}
			\sT_{[\rho]}\sM(\Gamma,\sG)&=\sT_{[\rho]}\sA(\Gamma,\sG)=\sT_\rho\sHom_{\sA}(\Gamma,\sG)/\sT_\rho(\sad(\sG)\rho)\\
			&\cong\mathsf{Z}^1_{\sad\circ\rho}(\Gamma,\fg)/\mathsf{B}^1_{\sad\circ\rho}(\Gamma,\fg)=\mathsf{H}^1_{\sad\circ\rho}(\Gamma,\fg).
		\end{align*}		
	\end{remark}

	\begin{proposition}\label{prop:prZero}
		Suppose $\{\rho_t\}_{t\in\R}\subset\sHom_{\sM}(\Gamma,\sG)$ is an analytic one parameter family with $\rho_0=\rho$ and
		\[\left.\frac{d}{dt}\right|_{t=0}\rho_t(\gamma)\rho(\gamma)^{-1}={U}(\gamma).\] 
		Moreover, suppose $\spr_{[\rho]}([{U}],[{U}])=0$ and $\left.\frac{d}{dt}\right|_{t=0}h_{\rho_t}=0$. Then \[[{U}]=0\in\mathsf{H}^1_{\sad\circ\rho}(\Gamma,\fg)\cong\sT_{[\rho]}\sM(\Gamma,\sG).\]
	\end{proposition}
	
	\begin{proof}
		Suppose $\spr_{[\rho]}([{U}],[{U}])=0$. We use Proposition \ref{prop.var} and Remark \ref{rem:prM} to obtain that $\left.\frac{d}{dt}\right|_{t=0}h_{\rho_t}f_{\rho_t}$ is Liv\v sic cohomologous to zero. Hence, by Theorem \ref{thm:liv} we deduce that for all $\gamma\in\cO$,
		\[\int_\gamma\left.\frac{d}{dt}\right|_{t=0}h_{\rho_t}f_{\rho_t}=0.\]
		Moreover, by our hypothesis $\left.\frac{d}{dt}\right|_{t=0}h_{\rho_t}=0$. It follows that for all $\gamma\in\cO$,
		\begin{align*}
		0&=\int_\gamma\left.\frac{d}{dt}\right|_{t=0}h_{\rho_t}f_{\rho_t}=h_\rho\int_\gamma\left.\frac{d}{dt}\right|_{t=0} f_{\rho_t}=h_\rho\left.\frac{d}{dt}\right|_{t=0} \frac{\alpha_{\rho_t}(\gamma)}{\ell(\gamma)}.
		\end{align*}
		Hence, we obtain $\dot{\alpha}(\rho(\gamma),{U}(\gamma))=0$ for all infinite order elements $\gamma\in\Gamma$. Finally, we use Theorem \ref{thm:infriggen} and conclude our result.
	\end{proof}

	\subsection{Margulis multiverses}
	
	In this subsection we study the constant entropy sections of the moduli space of Margulis spacetimes. We define a Margulis multiverse of entropy $k$ to be the section of the quotient moduli space of Margulis spacetimes with constant entropy $k$. We show that Margulis spacetimes having the same entropy bound a convex domain in the space of Margulis spacetimes with fixed linear parts and the restrictions of the pressure form on the Margulis multiverses are Riemannian. 
	
	\begin{lemma}\label{lem:enteq}
		Suppose $\varrho\in\sHom_{\sA}(\Gamma,\sH)$ and $u,v,w\in\mathsf{Z}^1_{\varrho}(\Gamma,\sT)$ with $(\varrho,u)\in\sHom_{\sM}(\Gamma,\sG)$. Then the following holds:
		\begin{enumerate}
			\item for all positive real numbers $a$,
			\[h(\varrho,au)=\frac{h(\varrho,u)}{a}\]
			\item for all real numbers $a,b$ with $a>0$,
			\[I((\varrho,au),(\varrho,bv))=\frac{b}{a}I((\varrho,u),(\varrho,v))\]
			\item $I((\varrho,u),(\varrho,v+w))=I((\varrho,u),(\varrho,v))+I((\varrho,u),(\varrho,w))$.
		\end{enumerate}
	\end{lemma}
	\begin{proof}
		We recall that 
		\[h(\varrho,au)=\lim_{T\to\infty}\frac{1}{T}\log|R_T(\varrho,au)|\] 
		where $R_T(\varrho,au)=\{\gamma\in\cO\mid\alpha_{(\varrho,au)}(\gamma)\leqslant T\}$. Moreover, we have
		\[\alpha{(\varrho(\gamma),au(\gamma))}=a\alpha(\varrho(\gamma),u(\gamma))\] 
		for all $\gamma\in\Gamma$. Hence,
		\begin{align*}
		R_T(\varrho,au)&=\{\gamma\in\cO\mid\alpha_{(\varrho,au)}(\gamma)\leqslant T\}=\{\gamma\in\cO\mid a\alpha_{(\varrho,u)}(\gamma)\leqslant T\}\\
		&=\left\{\gamma\in\cO\mid \alpha_{(\varrho,u)}(\gamma)\leqslant \frac{T}{a}\right\}=R_{\frac{T}{a}}(\varrho,u),
		\end{align*}
		and we conclude our first equality by observing that
		\begin{align*}
		h(\varrho,au)&=\lim_{T\to\infty}\frac{1}{T}\log|R_T(\varrho,au)|=\lim_{T\to\infty}\frac{1}{T}\log|R_{\frac{T}{a}}(\varrho,u)|\\
		&=\frac{1}{a}\lim_{\frac{T}{a}\to\infty}\frac{1}{T/a}\log|R_{\frac{T}{a}}(\varrho,u)|=\frac{h(\varrho,u)}{a}.
		\end{align*}
		Similarly, we further observe that
		\begin{align*}
		I((\varrho,au),(\varrho,bv+cw))&=\lim_{T\to\infty}\frac{1}{|R_T(\varrho,au)|}\sum_{\gamma\in R_T(\varrho,au)}\frac{\alpha_{(\varrho,bv+cw)}(\gamma)}{\alpha_{(\varrho,au)}(\gamma)}\\
		&=\lim_{{\frac{T}{a}}\to\infty}\frac{1}{|R_{\frac{T}{a}}(\varrho,u)|}\sum_{\gamma\in R_{\frac{T}{a}}(\varrho,u)}\frac{b\alpha_{(\varrho,v)}(\gamma)+c\alpha_{(\varrho,w)}(\gamma)}{a\alpha_{(\varrho,u)}(\gamma)}\\
		&=\frac{1}{a}(bI((\varrho,u),(\varrho,v))+cI((\varrho,u),(\varrho,w))),
		\end{align*}
		and our result follows.
	\end{proof}
	Let $\sM(\Gamma,\sG)_k\defeq\sM(\Gamma,\sG)\cap h^{-1}(k)$. We note that $\sM(\Gamma,\sG)_k$ is called a \emph{Margulis multiverse of entropy} $k$.
	\begin{lemma}\label{lem:constent}
		 Suppose $\sM(\Gamma,\sG)_k$ is a Margulis multiverse of entropy $k$. Then $\sM(\Gamma,\sG)_k$ is a codimension $1$ analytic submanifold of $\sM(\Gamma,\sG)$. Moreover, 
		\[\sM(\Gamma,\sG)_1=\{[(L_\rho,h(\rho)u_\rho)]\mid[\rho]\in\sM(\Gamma,\sG)\}.\]
	\end{lemma}
	\begin{proof}
		Firstly, we use Remark \ref{rem:marginv} (3) to observe that $h(\rho)=h(g\rho g^{-1})$ for all $g\in\sG$ and by Lemma \ref{lem:enteq} we notice that $h(L_\rho,h(\rho)u_\rho)=1$. Hence, 
		\[\sM(\Gamma,\sG)_1=\{[(L_\rho,h(\rho)u_\rho)]\mid[\rho]\in\sM(\Gamma,\sG)\}.\]
		Moreover, using Lemma \ref{lem:analent} we obtain that $h$ is analytic. We notice that
		\[\left.\frac{d}{dt}\right|_{t=0}h(L_\rho,\frac{u_\rho}{1+t})=h(L_\rho,u_\rho)\left.\frac{d}{dt}\right|_{t=0}(1+t)\neq0.\]
		Hence, by Implicit function theorem we conclude that $\sM(\Gamma,\sG)_k$ is an analytic submanifold of $\sM(\Gamma,\sG)$ of codimension $1$.
	\end{proof}
	\begin{theorem}\label{thm:Riem}
		Suppose $\sM(\Gamma,\sG)_k$ is a Margulis multiverse of entropy $k$. Then the restriction of the pressure form \[\spr:\sT\sM(\Gamma,\sG)_k\times\sT\sM(\Gamma,\sG)_k\rightarrow\R\] 
		is a Riemannian metric for all $k>0$.
	\end{theorem}
	\begin{proof}
		The result follows from Lemma \ref{lem:constent} and Proposition \ref{prop:prZero}.
	\end{proof}
	\begin{corollary}
		The pressure form on $\sM(\Gamma,\sG)$ is non-negative definite with signature $\dim(\sM(\Gamma,\sG))-1$. 
		
		In fact, the degenerate direction of the pressure form at any $[\rho]\in\sM(\Gamma,\sG)$ is $[(0,u_\rho)]\in\sH^1_{\sad\circ\rho}(\Gamma,\fg)$.
	\end{corollary}
	\begin{proof}
		We use Theorem \ref{thm:Riem} and Lemma \ref{lem:constent} to observe that the pressure form is non-negative definite. Moreover, for $\rho_t=(L_\rho,(1+t)u_\rho)$ we use Lemma \ref{lem:enteq} and obtain
		\[\frac{h(\rho_t)}{h(\rho)}I(\rho,\rho_t)=\frac{(1+t)}{(1+t)}\frac{h(\rho)}{h(\rho)}I(\rho,\rho)=1.\]
		Therefore, it follows that $\spr_{[\rho]}([(0,u_\rho)],[(0,u_\rho)])=0$ and we obtain that the signature is, in fact, $\dim(\sM(\Gamma,\sG))-1$ and the degenerate direction of the pressure form at $[\rho]$ is precisely $[(0,u_\rho)]$.
	\end{proof}

	\begin{theorem}
		Suppose $\varrho\in\sHom_{\sA}(\Gamma,\sH)$ and $u_0,u_1\in\mathsf{Z}^1_{\varrho}(\Gamma,\sT)$ are such that $(\varrho,u_0),(\varrho,u_1)\in\sHom_{\sM}(\Gamma,\sG)$. Moreover, suppose $h(\varrho,u_0)=k=h(\varrho,u_1)$. Then for $t\in(0,1)$ the representations $\rho_t\defeq (\varrho,tu_1+(1-t)u_0)$ are well defined, $\rho_t\in\sHom_{\sM}(\Gamma,\sG)$ and the following holds:
		\[h(\rho_t)< k.\]
	\end{theorem}
	\begin{proof}
		We use Proposition \ref{prop:convcone} to deduce that $\rho_t$ is well defined and $\rho_t\in\sHom_{\sM}(\Gamma,\sG)$ for $t\in(0,1)$.
		
		Now we will show the entropy inequality. We recall that $\sT=\R^{2n+1}=\ft$. We observe that if $v\in\mathsf{Z}^1_{\varrho}(\Gamma,\sT)=\mathsf{Z}^1_{\varrho}(\Gamma,\ft)$ then $(0,v)\in\mathsf{Z}^1_{\sad\circ(\varrho,w)}(\Gamma,\fg)$ for any $w\in\mathsf{Z}^1_{\varrho}(\Gamma,\sT)$. Indeed, we compute and observe that for all $\gamma,\eta\in\Gamma$,
		\[\sad(\varrho(\gamma),w(\gamma))(0,v(\eta))+(0,v(\gamma))=(0,\varrho(\gamma)v(\eta)+v(\gamma)).\]
		
		Suppose $u,v\in\mathsf{Z}^1_{\varrho}(\Gamma,\sT)$ is such that $\{(\varrho,u+tv)\}_{t\in(-1,1)}\subset\sHom_{\sM}(\Gamma,\sG)$, the entropy $h(\varrho,u)=k$ and $[(0,v)]\in\mathsf{H}^1_{\sad\circ(\varrho,u)}(\Gamma,\fg)$ is a tangent vector to the submanifold $\sM(\Gamma,\sG)_k$ at the point $[(\varrho,u)]\in\sM(\Gamma,\sG)$. 
		
		Hence, we obtain $\left.\frac{d}{dt}\right|_{t=0}h(\varrho,u+tv)=0$. Furthermore, we use Lemma \ref{lem:enteq} to observe that
		\begin{align*}
			I((\varrho,u),(\varrho,u+tv))&=I((\varrho,u),(\varrho,u))+tI((\varrho,u),(\varrho,v))\\
			&=1+tI((\varrho,u),(\varrho,v)).
		\end{align*}
		Hence, it follows that
		\[\left.\frac{d^2}{dt^2}\right|_{t=0}I((\varrho,u),(\varrho,u+tv))=0.\]
		Also, we deduce that,
		\[\left.\frac{d^2}{dt^2}\right|_{t=0}h(\varrho,u+tv)I((\varrho,u),(\varrho,u+tv))=\left.\frac{d^2}{dt^2}\right|_{t=0}h(\varrho,u+tv).\]
		Now using Proposition \ref{prop:bcls} and Remark \ref{rem:prM} we get that \[\left.\frac{d^2}{dt^2}\right|_{t=0}h(\varrho,u+tv)\geqslant0\] 
		with equality if and only if $\spr_{[(\varrho,u)]}([(0,v)],[(0,v)])=0$. Moreover, as \[\left.\frac{d}{dt}\right|_{t=0}h(\varrho,u+tv)=0,\] 
		we use Proposition \ref{prop:prZero} and deduce that $\spr_{[(\varrho,u)]}([(0,v)],[(0,v)])=0$ if and only if  $(0,v)\in\mathsf{B}^1_{\sad\circ(\varrho,u)}(\Gamma,\fg)$ i.e. $v\in\mathsf{B}^1_\varrho(\Gamma,\sT)$. Hence, for $v\notin\mathsf{B}^1_\varrho(\Gamma,\sT)$ we have
		\[\left.\frac{d^2}{dt^2}\right|_{t=0}h(\varrho,u+tv)>0.\] 
		It follows that $h(\varrho,u)$ is a local minimum. Therefore, for $u_0,u_1\in\mathsf{Z}^1_{\varrho}(\Gamma,\sT)$ with $(\varrho,u_0),(\varrho,u_1)\in\sHom_{\sM}(\Gamma,\sG)$ such that $h(\varrho,u_0)=k=h(\varrho,u_1)$ we deduce that
		\[h(\varrho,tu_1+(1-t)u_0)<th(\varrho,u_1)+(1-t)h(\varrho,u_0)=tk+(1-t)k=k\]
		and we conclude our result.
	\end{proof}

	\bibliography{bibliography.bib}
	\bibliographystyle{alpha}
	
\end{document}